\documentclass[a4paper,11pt,reqno,noindent]{amsart}
\usepackage[centertags]{amsmath}
\usepackage{amsfonts,amssymb,amsthm,dsfont,cases,amscd,esint,enumerate, stmaryrd}
\usepackage[T1]{fontenc}
\usepackage[english]{babel}
\usepackage[applemac]{inputenc}
\usepackage{newlfont}
\usepackage{color}
\usepackage[body={17cm,21.5cm}, top = 1in, bottom = 1.5in, centering]{geometry} 
\usepackage{fancyhdr}
\usepackage{enumerate}

\theoremstyle{plain}
\newtheorem{theor}{Theorem}[section]
\newtheorem{lem}[theor]{Lemma}
\newtheorem{prop}[theor]{Proposition}
\newtheorem{cor}[theor]{Corollary}
\theoremstyle{definition}

\newtheorem{rem}[theor]{Remark}
\newtheorem{defin}[theor]{Definition}

\numberwithin{equation}{section}

\newcommand{\Xd}{\mathbb X}

\newcommand{\R}{\mathbb R}
\newcommand{\Z}{\mathbb Z}
\newcommand{\T}{\mathbb T}

\newcommand{\Pc}{\mathcal{P}}
\newcommand{\loc}{{\operatorname{loc}}}
\newcommand{\Id}{\operatorname{Id}}
\newcommand{\Ker}{\operatorname{Ker}}
\newcommand{\Ld}{{L}}
\newcommand{\Div}{{\operatorname{div}}}
\newcommand{\Sym}{{\operatorname{sym}}}
\newcommand{\3}{\operatorname{|\hspace{-0.4mm}|\hspace{-0.4mm}|}}
\newcommand{\bigN}{\operatorname{\big|\hspace{-0.4mm}\big|\hspace{-0.4mm}\big|}}
\newcommand{\BigN}{\operatorname{\Big|\hspace{-0.4mm}\Big|\hspace{-0.4mm}\Big|}}
\newcommand{\biggN}{\operatorname{\bigg|\hspace{-0.4mm}\bigg|\hspace{-0.4mm}\bigg|}}
\newcommand{\step}[1]{\noindent \textbf{Step} #1.}
\newcommand{\ddr}{\mathrm{d}}

\usepackage{tikz}
\usepackage[colorlinks,citecolor=black,urlcolor=black]{hyperref}

\title[Mean-field approximation, Gibbs relaxation, and cross estimates]{Mean-field approximation, Gibbs relaxation,\\and cross estimates}
\author[A. Bernou]{Armand Bernou}
\address[Armand Bernou]{Universit\'e de Lyon, Universit\'e Claude Bernard Lyon 1, Laboratoire de Sciences Actuarielle et Financi\`ere, 50 Avenue Tony Garnier, F-69007 Lyon, France}
\email{armand.bernou@univ-lyon1.fr}
\author[M. Duerinckx]{Mitia Duerinckx}
\address[Mitia Duerinckx]{Universit\'e Libre de Bruxelles, D\'epartement de Math\'ematique, 1050~Brussels, Belgium}
\email{mitia.duerinckx@ulb.be}

\begin{document}

\begin{abstract}
We study the propagation of chaos and relaxation to Gibbs equilibrium for a system of~$N$ classical Brownian particles with weak mean-field interactions. It is well known that propagation of chaos holds uniformly in time with rate~$O(N^{-1})$ and that Gibbs relaxation holds uniformly in~$N$ with exponential rate~$O(e^{-ct})$. We go one step further by establishing a cross estimate that simultaneously captures both effects: the joint deviation between chaos propagation and Gibbs relaxation is of order~$O(N^{-1}e^{-ct})$. In particular, for translation-invariant systems, this yields an accelerated propagation of chaos, with the mean-field approximation error at the level of the one-particle density improving from~$O(N^{-1})$ to~$O(N^{-1}e^{-ct})$. Our approach relies on a detailed analysis of the BBGKY hierarchy for correlation functions, and applies to both underdamped and overdamped Langevin dynamics with merely bounded interaction forces. In addition, we obtain new quantitative results on Gibbs relaxation and provide partial extensions beyond the weak interaction regime.
\end{abstract}

\maketitle
\setcounter{tocdepth}{1}
\tableofcontents
\allowdisplaybreaks

\bigskip\noindent
{\textbf{MSC-class:}} Primary 60K35; Secondary 35Q70, 82C22, 82C31, 35Q84, 91D30.  

\section{Introduction}
\subsection{General overview}
We consider the Langevin dynamics for a system of $N$ Brownian particles with weak mean-field interactions, moving in a confining potential in $\R^d$, $d\ge1$.
For the purpose of this introduction, we focus on the case of the underdamped dynamics, but all our results hold unchanged for the overdamped dynamics; see Section~\ref{sec:ext-overdamped} below.
More precisely, we consider the following system of coupled SDEs: for $1\le i\le N$,
\begin{equation}\label{eq:Langevin}
\left\{\begin{array}{ll}
\ddr X^{i,N}_t= V^{i,N}_t \, \ddr t, \qquad &t \ge 0, \\[1mm]
\ddr V^{i,N}_t = - \beta V^{i,N}_t \, \ddr t + \tfrac\kappa N\sum_{j=1}^NK(X^{i,N}_t,X_t^{j,N})\,\ddr t-\nabla A(X_t^{i,N})\ddr t +\sqrt2\ddr B_t^{i}, \qquad &t \ge 0, \\[2mm]
(X^{i,N}_t, V^{i,N}_t)|_{t=0} = (X^{i,N}_\circ, V^{i,N}_\circ) & 
\end{array}\right.
\end{equation}
where $\{(X^{i,N},V^{i,N})\}_{1 \le i \le N}$ is the set of particle positions and velocities in the phase space $\Xd:=\R^d\times\R^d$,
where~$\{B^i\}_i$ are independent and identically distributed $d$-dimensional Brownian motions,
where $K$ is a bounded force kernel,
\begin{equation}\label{eq:KLinfty}
K\in L^\infty(\R^{2d})^d,
\end{equation}
where $A$ is a uniformly convex confining potential,
\begin{equation}\label{eq:confinement-A}
\nabla^2A\ge\lambda\Id,\qquad\text{for some $\lambda>0$},
\end{equation}
where $\beta > 0$ is a given constant,
and where $\kappa > 0$ is a coupling parameter describing the strength of interactions.
We shall primarily assume $\kappa$ to be small enough (independently of $N$) to guarantee good properties of the mean-field dynamics --- in particular, the uniqueness of the invariant measure.
For the overdamped dynamics, however, we also establish partial results that remain valid for arbitrary coupling strengths; see Section~\ref{sec:extensions} below.

\subsubsection{Propagation of chaos}
In terms of a probability density~$F^N$ on the $N$-particle phase space~$\Xd^N$, we recall that the dynamics~\eqref{eq:Langevin} is (formally) equivalent to the Liouville equation
\begin{multline}\label{eq:Liouville_Langevin}
\partial_t F^N + \sum_{i=1}^N v_i \cdot \nabla_{x_i} F^N=  \sum_{i=1}^N \Div_{v_i}\big((\nabla_{v_i} + \beta v_i) F^N \big) \\[-3mm]
- \frac{\kappa}{N} \sum_{i,j=1}^N K(x_i-x_j) \cdot \nabla_{v_i} F^N + \sum_{i=1}^N \nabla_{x_i} A \cdot \nabla_{v_i} F^N.
\end{multline}
We assume particles to be exchangeable, which amounts to the symmetry of $F^N$ in its $N$ variables $z_i=(x_i,v_i)\in\Xd$, $1\le i\le N$, and we assume for simplicity that they are exactly chaotic initially, meaning that initial positions and velocities $\{Z^{i,N}_\circ := (X^{i,N}_\circ, V^{i,N}_\circ)  \}_{1\le i\le N}$ are independent and have a common density $\mu_\circ\in\Pc(\Xd)$: equivalently, this means that $F^N$ is initially tensorized,
\begin{equation}\label{eq:chaos-in}
F_t^N|_{t=0}=\mu_\circ^{\otimes N}.
\end{equation}
Note that this is not really restrictive in view of de Finetti's theorem.
We now consider a global weak solution $F^N\in C(\R_+;\Pc\cap\Ld^1(\Xd))$ of the Liouville equation~\eqref{eq:Liouville_Langevin} with these tensorized data.
In the regime of a large number $N\gg1$ of particles, a classical question then concerns the evolution of finite subsets of `typical' particles in the system, as described by marginals of~$F^N$,
\[F^{N,m}(z_1,\ldots,z_m)\,:=\,\int_{\Xd^{N-m}}F^N(z_1,\ldots,z_N)\,\ddr z_{m+1}\ldots \ddr z_N,\qquad 1\le m\le N.\]
Correlations between particles are expected to remain negligible over time, so that the initial chaotic behavior~\eqref{eq:chaos-in} would be approximately propagated: for any fixed $m\ge1$ and $t\ge0$,
\begin{equation}\label{eq:prop-chaos}
F^{N,m}_t- (F^{N,1}_t)^{\otimes m}\,\to\,0,\qquad\text{as $N\uparrow\infty$}.
\end{equation}
As is well-known, this propagation of chaos would imply the mean-field limit
\begin{equation}\label{eq:cvg_to_mu}
F^{N,m}_t\,\to\, \mu_t^{\otimes m}, \qquad\text{as~$N\uparrow\infty$},
\end{equation}
where $\mu_t$ is the solution of the non-linear Vlasov-Fokker-Planck equation
\begin{equation}
\label{eq:VFP}
\left\{\begin{array}{l}
\partial_t\mu + v \cdot \nabla_x \mu =\Div_v((\nabla_v + \beta v) \mu) - \kappa (K \ast \mu) \cdot \nabla_v \mu + \nabla_x A \cdot \nabla_v \mu, \qquad t\ge0,\\
\mu|_{t=0}=\mu_\circ,
\end{array}\right.
\end{equation}
with the short-hand notation $K\ast\mu(z):=\int_{\Xd} K(x,y)\mu(y,w)\,\ddr y \ddr w$ for $z = (x,v) \in \Xd$. This topic has been extensively investigated since the 1990s. 
Quantitative estimates for the mean-field approximation~\eqref{eq:cvg_to_mu} have attracted much interest in recent years; see e.g.~\cite{Bolley_2010} for $K$ Lipschitz, \cite{Lacker-21} for $K$ bounded, and~\cite{BJS-22} for $K$ singular.
Uniform-in-time estimates cannot be expected in general: indeed, the mean-field equation~\eqref{eq:VFP} may admit multiple invariant measures while the $N$-particle system is uniquely ergodic. Yet, uniform-in-time estimates have recently been obtained in various settings for which the mean-field equilibrium is unique~\cite{Bolley_2010,Monmarche_2017,Delarue_Tse_21,Guillin_2022,Lacker-LeFlem-23,Chen_2024}:
in particular, for $K$ Lipschitz and $\kappa$ small enough, the following optimal estimate was obtained in~\cite{Lacker-LeFlem-23}, for all $1\le m\le N$ and~$t\ge0$,
\begin{align}\label{eq:prop_chaos_Langevin_W}
\mathcal{W}_2(F^{N,m}_t, \mu_t^{\otimes m}) \lesssim \frac{m}{N},
\end{align}
where $\mathcal W_2$ is the $2$-Wasserstein distance.
Such a uniform-in-time estimate does not only rely on the propagation of initial chaos, but also requires some mechanism to ensure that correlations created by particle interactions are damped over time: this is typically a consequence of ergodic properties of the linearized mean-field operator, which indeed drives the correlation dynamics. On this matter, we also refer to~\cite{Bolley_2010, BDM_25}, where the damping of initial correlations is captured explicitly.

While propagation of chaos~\eqref{eq:prop_chaos_Langevin_W} amounts to estimating correlations in the particle system, it is natural to further investigate the size of higher-order correlations: in turn, this provides a hierarchy of higher-order corrections to the mean-field approximation~\eqref{eq:cvg_to_mu} (see Bogolyubov correction in~\cite{MD-21,BD_2024,DJ-25} and higher orders in~\cite{Hess_Childs_2023}).
The two-particle correlation function is simply defined as
\begin{equation}\label{eq:def-GN2}
G^{N,2} \,:=\, F^{N,2} - (F^{N,1})^{\otimes 2},
\end{equation}
which captures the defect to propagation of chaos at the level of two-particle statistics. Higher-order correlation functions $\{G^{N,m}\}_{2\le m\le N}$ describe finer deviations; their definition is postponed to Section~\ref{sec:preliminary}. Formal considerations lead to expect the following scaling,
\begin{equation}\label{eq:control_Gm}
G^{N,m} = O(N^{1-m}), \qquad 2 \le m \le N,
\end{equation}
up to some dependence on $m$.
In recent years, such estimates have been established in various settings~\cite{MD-21,Hess_Childs_2023,DJ-25}, also uniformly in time~\cite{BD_2024,Xie-24,BDM_25}.

\subsubsection{Relaxation to equilibrium}
An advantage of uniform-in-time estimates such as~\eqref{eq:prop_chaos_Langevin_W}, or~\eqref{eq:control_Gm}, is that they allow to connect mean-field approximation and relaxation to equilibrium; see Figure~\ref{fig1} below. Let us briefly recall what is known regarding relaxation.

At the level of the mean-field dynamics~\eqref{eq:VFP}, if $K$ is bounded and if $\kappa$ is small enough, it is well known that there is a unique steady state~$M$ and that the solution $\mu_t$
converges exponentially,
\begin{equation*}
\mathcal W_2(\mu_t,M)+\|\mu_t-M\|_{L^2(M^{-1})}\,\le\, Ce^{-c_0t},
\end{equation*}
see e.g.~\cite{Bolley_2010} or Lemma~\ref{lem:ergodic-strong2}(i) below.
In the special case of a force kernel $K$ deriving from a potential, $K(x,y) = -\nabla W(x-y)$,
the steady state can be characterized as the solution of the fixed-point mean-field Gibbs equation
\begin{equation}\label{eq:def-M-eq}
M\,=\,Z[M]^{-1}e^{-\beta H[M]},\qquad H[M](x,v):=\tfrac12|v|^2 + A(x) + \kappa W \ast M(x),
\end{equation}
where $Z[M]\in\R_+$ is the normalization constant ensuring $\int_{\Xd} M= 1$.
Note that this equation has indeed a unique solution provided $\kappa \beta \|W\|_{\Ld^\infty(\R^d)} < 1$. 

At the level of the $N$-particle system, for fixed $N$, the existence and uniqueness of the steady state~$M^N$ for the Liouville equation~\eqref{eq:Liouville_Langevin} follows e.g.\@ from~\cite{Stroock1997} (without any smallness condition on~$\kappa$). Although not explicit in general, it reduces to the usual Gibbs equilibrium if $K(x,y)=-\nabla W(x-y)$,
\begin{align}
\label{eq:Gibbs_eq}
\quad M^N:=(Z^N)^{-1}e^{-\beta H^N},\qquad H^N(z_1,\ldots,z_N):=\sum_{i=1}^N\Big(\tfrac12|v_i|^2+A(x_i)+\tfrac\kappa{2N}\sum_{j=1}^NW(x_i-x_j)\Big),
\end{align} 
where $Z^N$ is the normalization constant. Similarly as for marginals of $F^N$, marginals of $M^N$ will be denoted by $\{M^{N,m}\}_{1\le m\le N}$.
For $K$ Lipschitz and for $\kappa$ small enough, the uniform-in-$N$ exponential relaxation of the Liouville solution~$F^N$ toward equilibrium~$M^N$ was first established in~\cite[Remark~13]{Bolley_2010} using coupling techniques: specifically, for all $1 \le m \le N$ and~$t\ge0$,
\begin{equation}\label{eq:gibbs-relax}
\mathcal W_2(F^{N,m}_t,M^{N,m})\,\le \, C e^{-c_0t}\mathcal W_2(\mu_\circ^{\otimes m},M^{N,m})\,\le \, Cm^\frac12 e^{-c_0t},
\end{equation}
for some constants $C, c_0>0$ independent of $N,m$.
Hypocoercivity techniques~\cite{Villani_2009} ensure that the same must hold in~$L^2$ norms with equilibrium weight, but with constants a priori depending on~$N$.
More recently, based on uniform-in-$N$ log-Sobolev and Poincar\'e inequalities, it was shown in~\cite{Monmarche_2017,Guillin_2021} that exponential relaxation in weighted~$L^2$ norms and in relative entropy actually holds with a time constant $c_0$ independent of $N$, but a priori still with an $N$-dependent prefactor: more precisely, for all~$1\le m\le N$ and~$t\ge0$,
\begin{eqnarray}
\textstyle\int_{\Xd^m}F_t^{N,m}\log({F_t^{N,m}}/{M^{N,m}})&\le& CNe^{-c_0t},\nonumber\\
\|F_t^{N,m}-M^{N,m}\|_{L^2((M^{-1})^{\otimes m})}&\le&C^{N}e^{-c_0t}.\label{eq:relax-old}
\end{eqnarray}
These estimates guarantee relaxation in relative entropy only for $t\gg\log N$, and in weighted $L^2$ norms only for $t\gg N$, in contrast with the $\mathcal W_2$-estimate~\eqref{eq:gibbs-relax} for $t\gg1$. We will show in the sequel how to improve upon these results and remove the $N$-dependent prefactors.

\subsubsection{Cross error estimate}
The Bogolyubov correction to the mean-field description is the next-order contribution arising from pair correlations (see e.g.~\cite{MD-21,BD_2024,DJ-25}), but its precise effects are often difficult to interpret. For conservative systems, the Lenard-Balescu theory predicts that this correction describes the very slow relaxation of the system toward equilibrium (see e.g.~\cite{MD-21}). In the present setting of Brownian particles, however, diffusion already drives relaxation on an $O(1)$ timescale, which raises the question of what the corresponding Bogolyubov correction describes here. As we shall see, it simply accounts for the fact that the $N$-particle Gibbs equilibrium differs from the tensorized mean-field equilibrium: the correction captures the emergence of these nontrivial finite-$N$ equilibrium correlations in the long-time limit.
This observation naturally leads us to studying the interplay between the mean-field approximation and the long-time relaxation: more precisely, our main result consists of the following {\it `cross error' estimate},
\begin{equation}\label{eq:cross-intro}
F_t^{N,m}- \mu_t^{\otimes m}- M^{N,m}+M^{\otimes m}\,=\,O(C^mN^{-1}) \times O(C^me^{-c_0t}),
\end{equation}
which seems to be a new type of result in the field.
This hybrid bound simultaneously recovers the optimal uniform-in-time mean-field estimate~\eqref{eq:prop_chaos_Langevin_W} and yields a uniform-in-$N$ version of the exponential relaxation bound~\eqref{eq:relax-old}.
It thus quantifies how fast the finite-$N$ system approaches equilibrium while remaining close to the main-field description.

\subsubsection{Overview of main results}
In this work, we revisit several classical questions regarding propagation of chaos, correlation estimates, and relaxation to Gibbs equilibrium. Our analysis culminates in the proof of the cross error estimate~\eqref{eq:cross-intro}; see Theorem~\ref{thm:main-Langevin} below. The proof combines a particularly detailed hierarchical analysis together with quantitative relaxation properties of the mean-field equation.
Along the way, we also obtain a number of additional results of interest, both for Gibbs relaxation and for propagation of chaos:
\begin{enumerate}[---]
\item \emph{Uniform-in-$N$ Gibbs relaxation:}
For the Langevin dynamics~\eqref{eq:Langevin}, if the coupling constant $\kappa$ is small enough, we establish the first Gibbs relaxation estimate in $L^2$ that is fully uniform in $N$, thereby improving upon~\eqref{eq:relax-old}. This could be deduced from the cross error bound~\eqref{eq:cross-intro}, but we rather derive it directly through a new hierarchical argument. To the best of our knowledge, this also provides the first exponential relaxation result beyond the case of a Lipschitz force kernel $K$.
\smallskip\item \emph{Time-accelerated propagation of chaos:}
For the corresponding translation-invariant Langevin dynamics on the torus, one has $M^{N,1} \equiv M$. The cross error estimate~\eqref{eq:cross-intro} then reduces to
\begin{align*}
F^{N,1}_t - \mu_t  = O(N^{-1} e^{-c_0 t}),
\end{align*}
which yields an improved, time-accelerated rate for propagation of chaos on long timescales. To our knowledge, such a refined rate appears here for the first time in the literature.
\smallskip\item \emph{Strongly coupled setting:}
For the {\it overdamped} Langevin dynamics on the torus, we identify several situations in which the smallness assumption on the coupling constant $\kappa$ can be partially removed; see Section~\ref{sec:extensions}. This includes, in particular, the case when~$K$ is translation-invariant and conservative. In that setting, a slightly more delicate variant of our argument yields the same cross error estimate up to an infinitesimal corrections in $N$,
\begin{equation*}
F_t^{N,m}- \mu_t^{\otimes m}- M^{N,m}+M^{\otimes m}\, =\, O(C_mN^{-1} e^{-c_0t}) + O(C_mN^{-\infty}).
\end{equation*} 
As a consequence, we deduce a partial Gibbs relaxation result,
\[F^{N,m}-M^{N,m}\,=\,O(C_me^{-c_0t})+O(C_mN^{-\infty}),\]
which ensures the expected exponential relaxation at least for $1\ll t\lesssim\log N$. Although not complete, this seems to be the only available relaxation result in this setting.
In addition, we obtain the corresponding accelerated propagation of chaos, cf.~Remark~\ref{rem:improved_chaos_kappa},
\[ F^{N,m} - \mu_t^{\otimes m} \,=\, O(C_mN^{-1} e^{-c_0 t}) + O(C_mN^{-\infty}).\]
\end{enumerate}
Finally, we note that the present analysis can also be extended to Kac's model. In a forthcoming work, we obtain analogous cross error estimates in that setting, together with a new uniform-in-$N$ exponential relaxation result.

\subsection{Main results}
Before turning to our main results on the cross error, we first revisit classical questions concerning propagation of chaos, correlation estimates, and relaxation to Gibbs equilibrium. For a sufficiently small coupling constant~$\kappa$, we show that the earlier results on these matters can be slightly strengthened through a direct analysis of the BBGKY hierarchy, combined with bootstrap arguments as in~\cite{Xie-24} and standard hypocoercivity estimates.
Throughout, we use weighted $L^2$ norms with equilibrium weight $\omega\simeq M^{-1}$,
\begin{align}
\label{eq:norm_hypoco}
	\|h\|_{L^2(\omega^{\otimes m})} := \Big( \int_{\Xd^m} \omega^{\otimes m} |h|^2 \Big)^{\frac12}, \qquad \omega(x,v) := \exp\big(\beta(\tfrac12|v|^2+A(x))\big).
\end{align}
More precisely, our first proposition establishes two key estimates: (i) a uniform-in-time bound on correlation functions, and (ii) a uniform-in-$N$ Gibbs relaxation result.
Item~(i) extends the recent work~\cite{Xie-24} on the overdamped dynamics to the present underdamped setting, the adaptation relying only on hypocoercivity arguments.
Item~(ii) improves on the earlier work on the topic~\cite{Bolley_2010,Monmarche_2017,Guillin_2021} for small~$\kappa$: it requires merely bounded interaction forces~$K$ and yields exponential relaxation in~$L^2$ uniformly in~$N$, thereby removing the $N$-dependent prefactor in~\eqref{eq:relax-old}. The proof, postponed to Section~\ref{sec:preliminary}, relies on a new, self-contained hierarchical argument which, to the best of our knowledge, constitutes the first hierarchical approach to Gibbs relaxation.
We note, however, that the aforementioned works~\cite{Bolley_2010,Monmarche_2017,Guillin_2021} have managed to significantly relax the smallness condition on $\kappa$ when $K$ is Lipschitz, a feature that currently seems out of reach for our PDE-based method. In the overdamped case, a partial extension beyond the weak-interaction regime will be discussed in Section~\ref{sec:extensions}.

\begin{prop}\label{prop:correlations_L2_Langevin}
Let $K \in L^{\infty}(\R^{2d})^d$, let $A$ satisfy~\eqref{eq:confinement-A}, and let $\mu_\circ \in \mathcal{P}(\Xd) \cap L^2(\omega)$. Consider the unique global weak solution $F^N\in C(\R_+; \mathcal{P}(\Xd^N) \cap L^2(\omega^{\otimes N}))$ of the Liouville equation~\eqref{eq:Liouville_Langevin} with tensorized data~\eqref{eq:chaos-in}, and the unique weak solution $\mu \in C(\R_+; \mathcal{P}(\Xd) \cap L^2(\omega))$ of~\eqref{eq:VFP}. There exist $c_0> 0$ (only depending on $d,\beta,K,A$) and $\kappa_0,C>0$ (further depending on $\|\mu_\circ\|_{L^2(\omega)}$) such that the following holds for any $\kappa \in [0, \kappa_0]$.
\begin{enumerate}[(i)]
\item \emph{Uniform-in-time correlation estimates:} for all $1 \le m \le N$ and $t \ge 0$,
\begin{equation}
\label{eq:correl_strong_Langevin}
\|G^{N,m}_t\|_{L^2(\omega^{\otimes m})} \le (Cm)^m N^{1-m}.
\end{equation}
In particular, this implies propagation of chaos: for all $1 \le m \le N$ and $t\ge0$,
\begin{equation}
\label{eq:prop-est-Langevin}
\|F^{N,m}_t - \mu_t^{\otimes m}\|_{L^2(\omega^{\otimes m})} \le C^m  N^{-1}.
\end{equation}
\item \emph{Gibbs relaxation:} for all $1 \le m \le N$ and $t \ge 0$, 
\begin{equation}
\label{eq:Gibbs_relax_strong_norms}
	\|F^{N,m}_t - M^{N,m}\|_{L^2(\omega^{\otimes m})} \le C^m e^{-c_0 t}. 
\end{equation}
\end{enumerate}
\end{prop}

We now turn to the interplay between the mean-field approximation and the long-time relaxation, and give a precise statement of the announced cross error estimate~\eqref{eq:cross-intro}.

\begin{theor}[Cross mean-field/relaxation error]
\label{thm:main-Langevin}
Let $K \in L^{\infty}(\R^{2d})^d$, let $A$ satisfy~\eqref{eq:confinement-A}, and let $\mu_\circ\in\mathcal{P}(\Xd) \cap L^2(\omega)$. Consider the unique global weak solution $F^N\in C(\R_+; \mathcal{P}(\Xd^N) \cap L^2(\omega^{\otimes N}))$ of the Liouville equation~\eqref{eq:Liouville_Langevin} with tensorized data~\eqref{eq:chaos-in}, and the unique weak solution $\mu \in C(\R_+; \mathcal{P}(\Xd) \cap L^2(\omega))$ of~\eqref{eq:VFP}.
There exist $c_0> 0$ (only depending on $d,\beta,K,A$) and $\kappa_0,C>0$ (further depending on $\|\mu_\circ\|_{L^2(\omega)}$) such that given~$\kappa\in[0,\kappa_0]$ we have for all $1 \le m \le N$ and $t \ge 0$,
\begin{equation}\label{eq:cross-err}
\big\| F^{N,m}_t-\mu_t^{\otimes m}-M^{N,m}+M^{\otimes m} \big\|_{L^2(\omega^{\otimes m})}  \,\le\, C^mN^{-1}e^{-c_0t}.
\end{equation}
In addition, denoting by $C^{N,m}$ the $m$-particle correlation function associated with Gibbs equilibrium~$M^N$, we have for all $1\le m\le N$ and~$t\ge0$,
\begin{equation}\label{eq:cross-err-cor}
\|G_t^{N,m}-C^{N,m}\|_{L^2(\omega^{\otimes m})}\,\le\,(Cm)^mN^{1-m}e^{-c_0t}.
\end{equation}
\end{theor}

\begin{center}
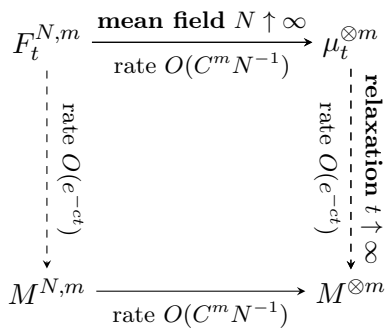
\begin{figure}
\begin{tikzpicture}[>=stealth, node distance=2.5cm, auto]
\node (A) at (0, 3.3) {$F^{N,m}_t$};
\node (B) at (4, 3.3) {$\mu_t^{\otimes m}$};
\node (C) at (0, 0) {$M^{N,m}$};
\node (D) at (4, 0) {$M^{\otimes m}$};
\path[->] (A) edge node[midway, below] {{\footnotesize rate $O(C^mN^{-1})$}} (B);
\path[->] (A) edge node[midway, above] {{\footnotesize {\bf mean field} $N\uparrow\infty$}} (B);
\path[->, dashed] (A) edge node[sloped, anchor=center, above] {{\footnotesize rate $O(e^{-ct})$}} (C)  ;
\path[->, dashed] (B) edge node[sloped, anchor=center, above]{{\footnotesize {\bf relaxation} $t\uparrow\infty$}} (D);
\path[->, dashed] (B) edge node[sloped, anchor=center, below]{{\footnotesize rate $O(e^{-ct})$}} (D);
\path[->] (C) edge node[midway, below]{{\footnotesize rate $O(C^mN^{-1})$}} (D);
\end{tikzpicture}
\caption{This diagram displays the estimates for mean-field and relaxation errors in $L^2(\omega^{\otimes m})$ norms. Our main result in this work is that the estimates for mean-field and relaxation errors further hold relatively to one another in form of a cross error estimate, cf.~\eqref{eq:cross-err}.
}\label{fig1}
\end{figure}
\end{center}

The above result is easily adapted to the corresponding particle system on the torus~$\T^d$, then choosing for instance $A\equiv0$.
In this case, the existence and uniqueness of the $N$-particle steady state~$M^N$ follows again from~\cite{Stroock1997}.
If the force kernel derives from a potential, $K(x,y)=-\nabla W(x-y)$,
then we find by symmetry that the first marginal~$M^{N,1}$ and the mean-field equilibrium $M$ both coincide with Maxwellian equilibrium,
\[M^{N,1}(x,v)\,=\,M(x,v) \,=\, Ce^{-\frac12\beta |v|^2}, \qquad (x,v) \in \T^d \times \R^d.\] 
In this translation-invariant setting, for $m=1$, the cross estimates of Theorem~\ref{thm:main-Langevin} thus lead to the following striking improvement of the usual uniform-in-time mean-field estimate~\eqref{eq:prop-est-Langevin}.

\begin{cor}[Accelerated propagation of chaos for translation-invariant systems]\label{cor:MFTorus-Langevin}
Consider the above translation-invariant setting with $\Xd = \T^d \times \R^d$, $A\equiv0$, $K(x,y) = -\nabla W(x-y)$ for some $W\in W^{1,\infty}(\T^{d})$, and let now $\omega(x,v)=\exp({\frac12\beta|v|^2})$. Given $\mu_\circ \in \Pc(\Xd) \cap L^2(\omega)$,
consider the unique global weak solution $F^N\in C(\R_+;\Pc(\Xd^N)\cap L^2(\omega^{\otimes N}))$ of the corresponding Liouville equation~\eqref{eq:Liouville_Langevin} on $\Xd^N$ with tensorized data~\eqref{eq:chaos-in}, and the unique weak solution $\mu \in C(\R_+; \mathcal{P}(\Xd) \cap L^2(\omega))$ of the corresponding mean-field equation~\eqref{eq:VFP} on $\Xd$. There exist $c_0 > 0$ (only depending on $d, \beta, K$) and $\kappa_0, C >0$ (further depending on $\|\mu_\circ\|_{L^2(\omega)}$) such that for any~$\kappa \in [0,\kappa_{0}]$ we have for all $t \ge 0$
\[\|F_t^{N,1}-\mu_t\|_{L^2(\omega)}\,\le\,C N^{-1} e^{-c_0 t}.\]
\end{cor}

\subsection{Extension~1: overdamped Langevin dynamics}\label{sec:ext-overdamped}

The above results are easily adapted to the overdamped Langevin dynamics, that is, to the following system of coupled SDEs: for $1\le i\le N$,
\begin{equation}\label{eq:overdamped}
\left\{\begin{array}{l}
\ddr X^{i,N}_t=\tfrac\kappa N\sum_{j=1}^NK(X^{i,N}_t,X_t^{j,N})\,\ddr t-\nabla A(X_t^{i,N})\ddr t +\sqrt2\ddr B_t^{i}, \qquad t \ge 0, \\[2mm]
X^{i,N}_t|_{t=0} = X^{i,N}_\circ,
\end{array}\right.
\end{equation}
where $(X^{i,N})_{1 \le i \le N}$ is now the set of particle positions in $\R^d$,
where~$\{B^i\}_i$ are independent and identically distributed $d$-dimensional Brownian motions,
where $K,A$ are as in~\eqref{eq:KLinfty}--\eqref{eq:confinement-A}, and where as before we shall take $\kappa>0$ to be small enough (independently of~$N$).
In terms of a probability density~$F^N$ on the $N$-particle space~$(\R^d)^N$, the corresponding Liouville equation reads
\begin{equation}\label{eq:Liouville_overdamped}
\partial_tF^N = \sum_{i=1}^N\Div_{x_i} \big((\nabla_{x_i}+\nabla A(x_i))F^N\big) - \kappa \sum_{i=1}^N \Div_{x_i} \bigg(\frac{1}{N} \sum_{j=1}^N K(x_i,x_j)F^N\bigg),
\end{equation}
and we assume that the distribution is initially tensorized, for some $\mu_\circ\in\Pc(\R^d)$,
\begin{equation}\label{eq:init_tensorization_overdamped}
F_t^N|_{t=0}=\mu_\circ^{\otimes N}.
\end{equation}
In this case, the propagation of chaos implies the mean-field limit
\begin{align}\label{eq:cvg_overdamped}
F^{N,m}_t \to\mu_t^{\otimes m}, \qquad \text{as $N \uparrow \infty$},
\end{align}
where $\mu_t$ is now the solution of the McKean-Vlasov equation
\begin{equation}\label{eq:MKV}
\left\{\begin{array}{l}
\partial_t \mu = \Div((\nabla + \nabla A) \mu) - \kappa \Div_x( (K \ast\mu) \mu),\qquad t\ge0,\\
\mu|_{t=0}=\mu_\circ.
\end{array}\right.
\end{equation}
Our analysis is easily repeated in this setting, up to adapting the notation accordingly: in particular, we now set
\begin{equation}\label{eq:notation-overdamped}
\Xd=\R^d,\qquad\|h\|_{L^2(\omega^{\otimes m})}=\Big(\int_{(\R^d)^m}\omega^{\otimes m}|h|^2\Big)^\frac12,\qquad\omega(x)=\exp(A(x)).
\end{equation}

Quantitative and uniform-in-time estimates for the mean-field approximation~\eqref{eq:cvg_overdamped} have been obtained under various conditions, see e.g.~\cite{Malrieu-03,Cattiaux-Guillin-Malrieu-08,DEGZ-20,Delarue_Tse_21,Guillin_2021b,BJS-22,Lacker-LeFlem-23,Rosenzweig-Serfaty-21,Rosenzweig-Serfaty-23}. For~$K$ bounded and $\kappa$ small enough, optimal uniform-in-time $L^2$ estimates on higher-order correlations $\{G^{N,m}\}_{2\le m\le N}$ were further obtained in~\cite{Xie-24}, establishing Proposition~\ref{prop:correlations_L2_Langevin}(i) in this setting.

For the mean-field equation~\eqref{eq:MKV}, uniqueness of the steady state $M$ follows from standard parabolic theory for $K$ bounded and $\kappa$ small enough, and the convergence is exponential; see e.g.~Lemma~\ref{lem:ergodic-strong} below. At the level of the $N$-particle system, the existence and uniqueness of the steady state~$M^N$ for the Liouville equation~\eqref{eq:Liouville_overdamped} follows e.g.\@ from~\cite[Section 3.3]{Kulik_2017}.
For $K$ Lipschitz and $\kappa$ small enough, as in the underdamped setting~\eqref{eq:gibbs-relax}, the coupling argument in~\cite{Bolley_2010} ensures the uniform-in-$N$ exponential relaxation of the Liouville solution $F^N$ in $2$-Wasserstein distance.
Convergence in relative entropy follows from the uniform-in-$N$ log-Sobolev inequality in~\cite{Otto_2007}.
More recently, in~\cite{Guillin_Liu_2022}, the smallness condition on~$\kappa$ was substantially weakened and exponential convergence in $L^2$ was deduced from a uniform-in-$N$ Poincar\'e inequality. However, as in the underdamped setting~\eqref{eq:relax-old}, estimates in relative entropy and in $L^2$ were obtained only up to $N$-dependent prefactors and they all required~$K$ to be Lipschitz.
For $\kappa$ small enough, the proof of Proposition~\ref{prop:correlations_L2_Langevin}(ii) is easily adapted to the present overdamped setting and thus again yields a partial improvement of previous work.

We turn to our main result on the cross error: it can also be adapted directly and takes on the following guise.

\begin{theor}[Cross mean-field/relaxation error]
\label{thm:main-overdamped}
Let $K\in L^\infty(\R^{2d})^d$, let $A$ satisfy~\eqref{eq:confinement-A}, and let $ \mu_\circ \in \Pc(\R^d)\cap L^2( \omega)$. Consider the unique global weak solution $ F^N\in C(\R_+;\Pc(\R^{dN})\cap\Ld^2( \omega^{\otimes N}))$ of the Liouville equation~\eqref{eq:Liouville_overdamped} with tensorized data~\eqref{eq:init_tensorization_overdamped}.
There exist $c_0> 0$ (only depending on $d,K,A$) and $\kappa_0,C>0$ (further depending on~$\| \mu_\circ\|_{L^2( \omega)}$) such that given~$\kappa\in[0,\kappa_0]$ we have for all $1 \le m \le N$ and $t \ge 0$,
\[ \big\|  F^{N,m}_t-  \mu_t^{\otimes m}- M^{N,m}+  M^{\otimes m} \big\|_{L^2( \omega^{\otimes m})}  \,\le\, C^mN^{-1}e^{-c_0t}.\]
\end{theor}

As before, this result holds unchanged for the corresponding system of particles on the torus~$\T^d$ with $A \equiv 0$.  If the force kernel is translation invariant, then we find by symmetry $M^{N,1}=M$, leading to the following strong improvement of the usual uniform-in-time mean-field estimate~\eqref{eq:prop-est-Langevin}.

\begin{cor}[Accelerated propagation of chaos for translation-invariant systems]\label{cor:MFTorus-Xie}
Let $\Xd=\T^d$, $A\equiv0$, and $K(x,y) = K_0(x-y)$ for some $K_0\in L^\infty(\T^{d})^d$. Given $\mu_\circ \in \Pc\cap L^2(\T^d)$, consider the unique global weak solution $F^N\in C(\R_+;\Pc\cap L^2(\T^{dN}))$ of the corresponding Liouville equation~\eqref{eq:Liouville_overdamped} on $\T^{dN}$ with tensorized data~\eqref{eq:init_tensorization_overdamped}, and the unique weak solution $\mu\in C(\R_+;\Pc\cap\Ld^2(\T^d))$ of the corresponding mean-field equation~\eqref{eq:MKV} on $\T^d$.
Then there exist $c_0> 0$ (only depending on $d,K$) and $\kappa_0,C>0$ (further depending on $\|\mu_\circ\|_{L^2(\T^d)}$) such that given~$\kappa\in[0,\kappa_{0}]$ we have for all $t \ge 0$,
\[\|F_t^{N,1}-\mu_t\|_{L^2(\T^d)}\,\le\,CN^{-1}e^{-c_0t}.\]
\end{cor}

\subsection{Extension~2: beyond the weak interaction regime}\label{sec:extensions}

The proof of Theorem~\ref{thm:main-Langevin} proceeds by a detailed analysis of the BBGKY hierarchy and relies on two ingredients: ergodic estimates for mean-field operators and uniform-in-time correlation estimates. Our approach also uses multiple bootstrap arguments for small~$\kappa$, but this can somehow be avoided: rather using direct inductions, the same results could be obtained up to~$O(N^{-\infty})$ errors.
In our recent work~\cite{BD_2024}, building on the master equation formalism from~\cite{Delarue_Tse_21}, we actually showed that, in negative Sobolev norms rather than in $L^2$, uniform-in-time correlation estimates can themselves be deduced from suitable ergodic estimates for mean-field operators (provided the interaction kernel is smooth). More precisely, we need the unique ergodicity of the mean-field dynamics and the relaxation of the linearized mean-field evolution (linearized at the mean-field solution). In weak norms and up to $O(N^{-\infty})$ errors, we therefore expect to adapt all our arguments, without any smallness condition on~$\kappa$, whenever such ergodic estimates are available.
In the torus $\T^d$ with~$A\equiv0$, taking advantage of the rich literature on the McKean-Vlasov equation~\eqref{eq:MKV}, this allows to cover the overdamped dynamics with any $\kappa\ge0$ in the following two cases:
\begin{enumerate}[---]
\item \emph{Case~1:} $K$ is translation-invariant and conservative; more precisely, $K(x,y)=K_0(x-y)$ for some~$K_0\in C^\infty(\T^d)^d$ that is odd and divergence-free.
\smallskip\item \emph{Case~2:} $K$ derives from a smooth, coordinate-wise even, positive-definite potential $W$ (so-called $H$-stable potential); more precisely, $K(x,y)=-\nabla W(x-y)$ for some $W\in C^\infty(\T^d)$ that is even in each coordinate and has Fourier coefficients $\hat W(k)\ge0$ for all $k\in\Z^d$.\footnote{In fact, this non-negativity condition can even be relaxed to $1 + 2\kappa\inf_{k\in\Z^d}\hat W(k) >0$.} Note that this is the only framework in this work for which a distinction can be drawn between attractive and repulsive interactions; see also example~\eqref{ex:Coulomb} in Section~\ref{sec:exs} below.
\end{enumerate}
In both cases, it can be shown that the mean-field steady state is unique and coincides with the Lebesgue measure for any $\kappa\ge0$, and ergodic properties of linearized mean field were established in~\cite[Sections~3.6.2--3.6.3]{Delarue_Tse_21} (see in particular~\cite{Carrillo_2023} for Case~2).
A suitable adaptation of our analysis then leads to the following cross error estimates in weak norms up to $O(N^{-\infty})$. This implies in particular a uniform-in-$N$ Gibbs relaxation result up to $O(N^{-\infty})$, cf.~\eqref{eq:relax-weak} below, which is, to our knowledge, the first such result for those systems.
The proof is postponed to Section~\ref{sec:extension-app}.

\begin{theor}\label{th:non-pert}
Consider Cases~1 and~2 above on the torus $\T^d$ with $A\equiv0$.
Given $\mu_\circ\in\Pc\cap C^\infty(\T^d)$,
consider the unique global weak solution $F^N\in C(\R_+;\Pc\cap C^\infty(\T^{dN}))$ of the corresponding Liouville equation~\eqref{eq:Liouville_overdamped} on $\T^{dN}$ with tensorized data~\eqref{eq:init_tensorization_overdamped}, and the unique weak solution $\mu\in C(\R_+;\Pc\cap C^\infty(\T^d))$ of the corresponding mean-field equation~\eqref{eq:MKV} on $\T^d$.
Then, in both Cases~1 and~2, for any~$\kappa\ge0$, we have for all $1\le m\le N$ and $t,\theta\ge0$,
\begin{equation}\label{eq:cross-weakre}
\big\|F_t^{N,m}-\mu_t^{\otimes m}-M^{N,m}+M^{\otimes m}\big\|_{W^{-\ell_{m,\theta},1}(\T^{dm})}\,\le\,C_{m,\theta}\big(N^{-1}e^{-c_0t}+N^{-\theta}\big),
\end{equation}
for some $c_0>0$ only depending on $d,K$, some regularity exponent $\ell_{m,\theta}$ only depending on $m,\theta$, and some constant $C_{m,\theta}$ further depending on $d,K,\mu_\circ$.
In particular, this implies the following Gibbs relaxation result: for all $1\le m\le N$ and $t,\theta\ge0$,
\begin{equation}\label{eq:relax-weak}
\|F_t^{N,m}- M^{N,m}\|_{W^{-\ell_{m,\theta},1}(\T^{dm})}\,\le\, C_{m,\theta}\big(e^{-c_0t}+N^{-\theta}\big).
\end{equation}
\end{theor}

\begin{rem}
\label{rem:improved_chaos_kappa}
In Case~1, as $M^N=1$, the cross estimate~\eqref{eq:cross-weakre} further implies the following accelerated propagation of chaos, for all $1\le m\le N$ and~$t,\theta\ge0$,
\[\|F_t^{N,m}-\mu_t^{\otimes m}\|_{W^{-\ell_{m,\theta},1}(\T^{dm})}\,\le\,C_{m,\theta}\big(N^{-1}e^{-c_0t}+N^{-\theta}\big).\]
In Case~2, the same holds for $m=1$ as $M^{N,1}=M=1$.
\end{rem}

\subsection{Examples}\label{sec:exs}
We briefly describe some applications of our results.
In item~\eqref{ex:opinion}, we give an example with merely bounded interaction forces, for which even our uniform-in-$N$ Gibbs relaxation result in Proposition~\ref{prop:correlations_L2_Langevin}(ii) is new.

\begin{enumerate}[(a)]
\smallskip\item\emph{Mollified Coulomb interactions:}
\label{ex:Coulomb}\\
Coulomb interactions in dimension $d\ge2$ are of course excluded in this work due to their singularity, but we can consider mollified versions (see~\cite{Rosenzweig-Serfaty-21,Rosenzweig-Serfaty-23} otherwise). For instance, on $\R^d$ with quadratic confinement $A=\frac12|x|^2$, we can consider the particle system with interactions given by
\[K(x,y)=\frac{x-y}{|x-y|^d+\sigma},\]
for some fixed $\sigma>0$.
In this setting,
both for the underdamped and for the overdamped Langevin dynamics,
all our results apply: for weak interactions, that is, for $\kappa$ small enough, Proposition~\ref{prop:correlations_L2_Langevin} yields uniform-in-time correlation estimates and uniform-in-$N$ Gibbs relaxation, and Theorem~\ref{thm:main-Langevin} further yields new estimates for the cross mean-field/relaxation error.

\smallskip\noindent
On the torus $\T^d$ with $A\equiv0$, mollified Coulomb interactions can be defined by convolving the periodic Coulomb potential $W_0$,
\[K(x,y)=-\nabla W(x-y),\qquad W=\Gamma_\sigma\ast W_0,\]
for instance with the periodized Gaussian kernel $\Gamma_\sigma(x)=(2\pi\sigma^2)^{-d/2}\sum_{n\in\Z^d}\exp(-\frac1{2\sigma^2}|x+n|^2)$, for some fixed~$\sigma>0$.
For $\kappa$ small enough, the same results hold in this setting, and we recall that Corollary~\ref{cor:MFTorus-Langevin} then further yields an improved convergence rate for the one-particle density (see Corollary~\ref{cor:MFTorus-Xie} in the overdamped setting).
In addition, it is easily checked that $K$ satisfies the assumptions of Case~2 in Section~\ref{sec:extensions}: hence, in the overdamped setting, without any smallness requirement on $\kappa>0$, Theorem~\ref{th:non-pert} yields a uniform-in-$N$ Gibbs relaxation result and cross error estimates up to $O(N^{-\infty})$ in negative Sobolev norms.

\smallskip
\item \emph{Noisy Hegselmann-Krause model for opinion dynamics:}\label{ex:opinion}\\
This model, introduced in~\cite{HK_02}, describes the evolution of the opinion of agents only influenced by their neighbors within a radius. The long-term behavior of the corresponding mean-field equation was studied in~\cite{Carrillo_2019}. The model takes form of the overdamped Langevin dynamics~\eqref{eq:overdamped}, say on the torus $\T^d$ with $A \equiv 0$ for simplicity, with interactions given by
\[K(x,y) = - \nabla W(x-y),\qquad W(x) =\left\{\begin{array}{lll}
(|x|-r)^2&:&|x|\le r,\\
0&:&|x|> r,
\end{array}\right.
 \]
for some fixed radius $r>0$.
In this setting, for $\kappa$ small enough, uniform-in-time propagation of chaos and correlation estimates follow from~\cite{Xie-24}. Corollary~\ref{cor:MFTorus-Xie} further provides an improved convergence rate for the one-particle density. Next, since $K$ is not Lipschitz, we obtain the first uniform-in-$N$ Gibbs relaxation result, cf.\@ Proposition~\ref{prop:correlations_L2_Langevin}(ii). Finally, Theorem~\ref{thm:main-overdamped} provides new estimates on the cross mean-field/relaxation error.

\smallskip\noindent
The same conclusions also hold for the model introduced in~\cite{Barre2017} in the context of interacting dynamical networks, which is an extension of the above with
\[W(x) = \left\{ 
\begin{array}{lll}
    (|x| - \ell)^2 - (r-\ell)^2&:& |x| \le r, \\
    0&:&|x| > r,
\end{array}
\right.
\]
for some parameters $r\ge\ell>0$.
\end{enumerate}

\subsection*{Notation}
\begin{enumerate}[---]
\item
We denote by $C \ge 1$ any constant that only depends on the space dimension $d$, on $\|K\|_{L^\infty}$, on the convexity bound $\lambda$ for $A$ in~\eqref{eq:confinement-A}, on $\|\mu_\circ\|_{L^2(\omega)}$, and possibly on other controlled quantities when specified.
We use the notation $\lesssim$ for $\le C \times $ up to such a multiplicative constant $C$. We add subscripts to $C, \lesssim$ to indicate dependence on other parameters.
\smallskip\item For an integer $m\ge1$, we write $\llbracket m\rrbracket := \{1,\ldots, m\}$, and for an index set $P\subset\llbracket N\rrbracket$ we write $z_P:=(z_{i_1},\ldots,z_{i_\ell})$ if $P=\{i_1,\ldots,i_\ell\}$.
\end{enumerate}

\subsection*{Plan of the paper}
Section~\ref{sec:preliminary} provides several preliminary definitions and results, and contains in particular a short proof of Proposition~\ref{prop:correlations_L2_Langevin}.
Section~\ref{sec:thm-Langevin} is devoted to the proof of Theorem~\ref{thm:main-Langevin} and Corollary~\ref{cor:MFTorus-Langevin}. Next, Section~\ref{sec:overdamped} describes the adaptation of our analysis to the overdamped dynamics. Finally, Section~\ref{sec:extension-app} is devoted to the proof of Theorem~\ref{th:non-pert}, which follows the same lines but requires some important adaptations, avoiding bootstrap arguments and using negative Sobolev norms.

\section{Preliminary}\label{sec:preliminary}
This section is devoted to several preliminary results, and in particular to the proof of Proposition~\ref{prop:correlations_L2_Langevin} (see Sections~\ref{subsec:Gibbs_relax} and~\ref{sec:corr_Langevin} below).
We focus here on the underdamped Langevin dynamics~\eqref{eq:Liouville_Langevin}. We recall $\Xd = \R^d \times \R^d$ and $\omega(z)=\exp(\beta(\frac12|v|^2+A(x)))$ for $z=(x,v)\in\Xd$.

\subsection{Mean-field relaxation}
Consider the following linear Fokker-Planck operator on the weighted space~$L^2(\omega)$,
\[ L_0h\,:=\,\Div_v((\nabla_v+\beta v)h)-v\cdot\nabla_xh+\nabla A(x)\cdot\nabla_vh.\]
The following ergodic estimates for the mean-field evolution~\eqref{eq:VFP} and for the Fokker-Planck semigroup~$\{e^{t L_0}\}_{t\ge0}$ follow from standard hypocoercivity theory~\cite{Villani_2009,Dolbeault_Hypocoercivity_2015}.

\begin{lem}\label{lem:ergodic-strong2}
Let $K\in L^\infty(\R^{2d})^d$ and let the confining potential $A$ satisfy~\eqref{eq:confinement-A}. There exist $\kappa_0,c_0>0$ (only depending on $d,\beta, K,A$) such that the following hold for all $\kappa\in[0,\kappa_0]$.
\begin{enumerate}[(i)]
\item Given $\mu_\circ\in \Pc(\Xd)\cap L^2(\omega)$, the solution $\mu\in C(\R_+;\Pc(\Xd)\cap L^2(\omega))$ of the mean-field equation~\eqref{eq:VFP} satisfies for all $t\ge0$,
\begin{align}
    \label{eq:cvg-MF}
\|\mu_t- M\|_{L^2(\omega)}\,\lesssim\,e^{-c_0t}\|\mu_\circ- M\|_{L^2(\omega)}.
\end{align}
\item For all $h\in C^\infty_c(\Xd)$, we have for all $t\ge0$,
\begin{equation}\label{eq:estimL-0}
e^{2c_0t}\|\pi^\bot e^{t L_0}h\|_{L^2(\omega)}^2+\int_0^te^{2c_0s}\|(\nabla_v+\beta v) e^{s L_0}h\|_{L^2(\omega)}^2 \, \ddr s\,\lesssim\,\|\pi^\bot h\|_{L^2(\omega)}^2,
\end{equation}
where $\pi^\bot$ stands for the orthogonal projection onto $\Ker(L_0)^\bot$,
\[\pi^\bot:=\Id-\pi,\qquad \pi h:=M_0\int_\Xd h,\qquad M_0:=\tfrac{\omega^{-1}}{\int_\Xd\omega^{-1}}.\]
In particular, by duality, for all $g\in L^2_\loc(\R_+;L^2(\Xd)^d)$,
\begin{equation}\label{eq:estimL-dual}
\Big\|\int_0^te^{(t-s) L_0}\Div_v(g(s))\,\ddr s\Big\|_{L^2(\omega)}^2\,\lesssim\,\int_0^te^{-2c_0(t-s)}\|g(s)\|_{L^2(\omega)}^2\,\ddr s.
\end{equation}
\end{enumerate}
\end{lem}

\begin{proof}
We start with a brief discussion of item~(ii).
By the hypocoercivity techniques in~\cite[Theorem~10]{Dolbeault_Hypocoercivity_2015} (see Remark~\ref{rem:hypocoerc} below),
there is $c_0>0$ (only depending on $d,\beta,A$) such that for all~$t \ge 0$,
\begin{equation}\label{eq:estimL-DMS}
\|\pi^\bot e^{t L_0}h\|_{L^2(\omega)}\,\lesssim\,e^{-c_0t}\|\pi^\bot h\|_{L^2(\omega)}.
\end{equation}
Using the energy identity for $e^{tL_0}h$,
\[\partial_t\|\pi^\bot e^{t L_0}h\|_{L^2(\omega)}^2+2\|(\nabla_v+\beta v) e^{t L_0}h\|_{L^2(\omega)}^2\le0,\]
we may then deduce from~\eqref{eq:estimL-DMS},
\[\partial_t\Big(e^{c_0t}\|\pi^\bot e^{t L_0}h\|_{L^2(\omega)}^2\Big)+2e^{c_0t}\|(\nabla_v+\beta v) e^{t L_0}h\|_{L^2(\omega)}^2\le c_0e^{c_0t}\|\pi^\bot e^{t L_0}h\|_{L^2(\omega)}^2\lesssim e^{-c_0t}\|\pi^\bot h\|_{L^2(\omega)}^2,\]
and thus after integration,
\begin{equation}\label{eq:estimL-0-pre}
e^{c_0t}\|\pi^\bot e^{t L_0}h\|_{L^2(\omega)}^2+2\int_0^te^{c_0s}\|(\nabla_v+\beta v) e^{s L_0}h\|_{L^2(\omega)}^2 \, \ddr s\,\lesssim\, \|\pi^\bot h\|_{L^2(\omega)}^2,
\end{equation}
that is,~\eqref{eq:estimL-0} (up to renaming $c_0$).
Next, by duality and integration by parts, we find
\begin{eqnarray*}
\lefteqn{\Big\|\int_0^te^{(t-s)L_0}\Div_v(g(s))\, \ddr s\,\Big\|_{L^2(\omega)}}\\
&=&\sup\bigg\{\int_0^t\Big(\int_\Xd\omega he^{(t-s)L_0}\Div_v(g(s))\Big) \ddr s:\|h\|_{L^2(\omega)}\le1\bigg\}\\
&=&\sup\bigg\{\int_0^t\Big(\int_\Xd\omega g(s)\cdot(\nabla_v+\beta v)e^{(t-s)L_0^*}h\Big)\, \ddr s:\|h\|_{L^2(\omega)}\le1\bigg\}\\
&\le&\Big(\int_0^te^{-c_0(t-s)}\|g(s)\|_{L^2(\omega)}^2 \, \ddr s \Big)^\frac12\sup\bigg\{\Big(\int_0^te^{c_0s}\|(\nabla_v+\beta v)e^{sL_0^*}h\|_{L^2(\omega)}^2 \, \ddr s\Big)^\frac12:\|h\|_{L^2(\omega)}\le1\bigg\}.
\end{eqnarray*}
Noting that~\eqref{eq:estimL-0-pre} also holds for $L_0$ replaced by its adjoint $L_0^*$ on $L^2(\omega)$, the conclusion~\eqref{eq:estimL-dual} follows (up to renaming~$c_0$ again).

For completeness, we include a short proof of item~(i) as a consequence of~(ii) for small $\kappa$. Let $\mu,\mu'\in C(\R_+;\Pc(\Xd)\cap\Ld^2(\omega))$ satisfy the mean-field equation~\eqref{eq:VFP} with initial data $\mu_\circ,\mu_\circ'\in\Pc(\Xd)\cap\Ld^2(\omega)$. Treating interactions perturbatively, Duhamel's formula for~\eqref{eq:VFP} yields
\[\mu_t=e^{tL_0}\mu_\circ-\kappa\int_0^te^{(t-s)L_0}\big((K\ast\mu_s)\cdot\nabla_v\mu_s\big)\ddr s,\]
and thus, applying the projection $\pi^\bot$, taking the norm, and applying the estimates of item~(ii) for the Fokker-Planck semigroup,
\begin{eqnarray*}
\|\pi^\bot\mu_t\|_{L^2(\omega)}^2
&\lesssim&e^{-2c_0t}\|\pi^\bot\mu_\circ\|_{L^2(\omega)}^2+\kappa^2\int_0^te^{-2c_0(t-s)}\|\mu_s(K\ast\mu_s)\|_{L^2(\omega)}^2\ddr s\\
&\lesssim&e^{-2c_0t}\|\pi^\bot\mu_\circ\|_{L^2(\omega)}^2+\kappa^2\int_0^te^{-2c_0(t-s)}\|\mu_s\|_{L^2(\omega)}^2\ddr s.
\end{eqnarray*}
Since $\mu=\pi^\bot\mu+M_0$, this yields
\begin{equation*}
\|\mu_t\|_{L^2(\omega)}^2
\,\le\,
C\|\mu_\circ\|_{L^2(\omega)}^2
+C\kappa^2\int_0^te^{-2c_0(t-s)}\|\mu_s\|_{L^2(\omega)}^2\ddr s.
\end{equation*}
For $\kappa$ small enough so that $C\kappa^2\le c_0$,
we deduce by Gronwall's inequality, for all $t\ge0$,
\begin{equation}\label{eq:apriori-mu}
\|\mu_t\|_{L^2(\omega)}\,\lesssim\,\|\mu_\circ\|_{L^2(\omega)}.
\end{equation}
Let us now examine the difference $\mu-\mu'$, which satisfies
\[(\partial_t-L_0)(\mu-\mu')=-\kappa K\ast\mu'\cdot\nabla_v(\mu-\mu')-\kappa K\ast(\mu-\mu')\cdot\nabla_v\mu.\]
Writing Duhamel's formula for this equation, taking the norms, and using again the estimates of item~(ii) with $\pi^\bot(\mu-\mu')=\mu-\mu'$, we get
\begin{eqnarray*} 
\lefteqn{\|\mu_t-\mu'_t\|_{L^2(\omega)}^2}\\
&\lesssim&e^{-2c_0t}\|\mu_\circ-\mu'_\circ\|_{L^2(\omega)}^2+\kappa^2 \int_0^te^{-2c_0(t-s)}\Big\|(\mu_s-\mu_s')(K\ast\mu_s')+\mu_s(K\ast(\mu_s-\mu_s'))\Big\|_{L^2(\omega)}^2\ddr s\\
&\lesssim&e^{-2c_0t}\|\mu_\circ-\mu'_\circ\|_{L^2(\omega)}^2+\kappa^2 \int_0^te^{-2c_0(t-s)}\big(1+\|\mu_s\|_{L^2(\omega)}\big)^2\|\mu_s-\mu_s'\|_{L^2(\omega)}^2\ddr s.
\end{eqnarray*}
Using the above a priori bound~\eqref{eq:apriori-mu} on $\mu$, Gronwall's inequality yields for $\kappa$ small enough, for all~$t\ge0$,
\[\|\mu_t-\mu'_t\|_{L^2(\omega)}^2\,\lesssim\,e^{-c_0t}\|\mu_\circ-\mu'_\circ\|_{L^2(\omega)}^2.\]
From this contraction property, we find that there is a unique steady state $M$, that it belongs to $L^2(\omega)$, and the conclusion~(i) follows.
\end{proof}

\begin{rem}\label{rem:hypocoerc}
For completeness, we emphasize that the assumptions in~\cite{Dolbeault_Hypocoercivity_2015} are indeed satisfied for the operator $L_0$ on $L^2(\omega)$ with~$A$ satisfying~\eqref{eq:confinement-A}, so we can apply their hypocoercivity results and deduce~\eqref{eq:estimL-DMS} above.
Assumption~(H1) in~\cite{Dolbeault_Hypocoercivity_2015} is simply the Gaussian Poincar\'e inequality, and assumption~(H3) is trivial for kinetic equations. Next, as $A$ is uniformly convex, we recall that the probability measure $\propto e^{-A}$ on $\R^d$ satisfies a Poincar\'e inequality, see e.g.~\cite[Section A.19]{Villani_2009}. By~\cite[Lemma~3]{Dolbeault_2009}, this ensures the validity of Assumption~(H2) in~\cite{Dolbeault_Hypocoercivity_2015}. Finally, given $w\in C^\infty_c(\R^d)$, for the solution $u$ of
\[e^{-A}u-\Div(e^{-A}\nabla u)=e^{-A}w,\qquad\text{in $\R^d$},\]
we find by multiple integrations by parts, using the summation convention on repeated indices,
\begin{eqnarray*}
\|\nabla^2u\|^2_{L^2(e^{-A})}
&=&-\int_{\R^d}(\nabla_{j}u)\nabla_i(e^{-A}\nabla_{ij}^2u)\\
&=&-\int_{\R^d}(\nabla_{j}u)\nabla_{ij}^2(e^{-A}\nabla_{i}u)-\int_{\R^d}(\nabla_{j}u)\nabla_{i}(e^{-A}(\nabla_jA)\nabla_{i}u)\\
&=&-\int_{\R^d}(\nabla_{j}u)(\nabla_{j}+\nabla_jA)\nabla_{i}(e^{-A}\nabla_{i}u)-\int_{\R^d}(\nabla_{i}u)(\nabla_{j}u)(\nabla_{ij}^2A)e^{-A}.
\end{eqnarray*}
Hence, using the equation for $u$ in form of $\nabla_i(e^{-A}\nabla_iu)=e^{-A}(u-w)$, and recalling the uniform convexity of $A$,
\begin{equation*}
\|\nabla^2u\|^2_{L^2(e^{-A})}
\,\le\,-\int_{\R^d}e^{-A}(\nabla_{j}u)\nabla_{j}(u-w)
\,\le\,\int_{\R^d}e^{-A}(\nabla_{j}u)(\nabla_{j}w).
\end{equation*}
Further integrating by parts and using the equation for $u$, we get
\begin{multline*}
\|\nabla^2u\|^2_{L^2(e^{-A})}
\,\le\,-\int_{\R^d}w\nabla_j(e^{-A}\nabla_{j}u)
\,=\,\int_{\R^d}we^{-A}(w-u)\\
\,\lesssim\,\|w\|_{L^2(e^{-A})}^2+\|u\|_{L^2(e^{-A})}^2
\,\lesssim\,\|w\|_{L^2(e^{-A})}^2.
\end{multline*}
By~\cite[Lemma~4]{Dolbeault_Hypocoercivity_2015}, this proves the validity of Assumption~(H4) in~\cite{Dolbeault_Hypocoercivity_2015}.
\end{rem}

\begin{rem}
We recall the following useful observation, see e.g.~\cite[Lemma~II.3.10]{Engel-Nagel}: if $(T_t)_{t\ge0}$ is a semigroup on a Banach space $(B,\|\cdot\|_B)$ and if there exist $P>0$ and $\omega\in\R$ such that we have $\|T_th\|_B\le Pe^{\omega t}\|h\|_B$ for all $t\ge0$ and $h\in B$, then there exists a modified norm $\3\cdot\3_B$ on $B$ that is Lipschitz-equivalent to $\|\cdot\|_B$ such that $\3 T_th\3_B\le e^{\omega t}\3 h\3_B$, where now the multiplicative constant is equal to $1$. Indeed, it suffices to consider the norm
\[\3 h\3_B\,:=\,\sup_{s\ge0}e^{-\omega s}\|T_sh\|_B,\]
which obviously satisfies, for all $h\in B$,
\[\|h\|_B\le\3h\3_B\le P\|h\|_B,\qquad\3T_th\3_B
\le e^{\omega t}\3 h\3_B.\]
Applying this observation to our setting in~\eqref{eq:estimL-DMS} above, with $L_0$ replaced by its adjoint $L_0^*$, we can similarly construct a norm $\3\cdot\3_{L^2(\omega)}^*$ on $L^2(\omega)$ such that
\begin{equation}\label{eq:equiv-norm}
\|h\|_{L^2(\omega)}\le\3h\3_{L^2(\omega)}^*\lesssim \|h\|_{L^2(\omega)},\qquad
\3\pi^\bot e^{tL_0^*}h\3_{L^2(\omega)}^*\le e^{-c_0t}\3\pi^\bot h\3_{L^2(\omega)}^*,
\end{equation}
where we used that the orthogonal projection $\pi^\bot$ commutes with the semigroup.
For $m\ge1$, we also define $\3\cdot\3_{L^2(\omega^{\otimes m})}$ as the associated injective cross norm on $\Ld^2(\omega^{\otimes m})$,
\begin{equation}\label{eq:equi-norm-m}
\3h\3_{L^2(\omega^{\otimes m})}\,:=\,\sup\bigg\{\int_{\Xd^m} \omega^{\otimes m}(\phi_1\otimes\ldots\otimes\phi_m) h\,:\,\3\phi_1\3_{L^2(\omega)}^*=\ldots=\3\phi_m\3_{L^2(\omega)}^*=1\bigg\},
\end{equation}
which is of course also Lipschitz equivalent to the standard norm on $L^2(\omega^{\otimes m})$.
\end{rem}

\subsection{BBGKY hierarchy for marginals}
In order to formulate the BBGKY hierarchy in a more convenient way, we first introduce some short-hand notation.

\begin{defin}\label{def:operator-L2-Langevin}
Consider a collection $\{h^m\}_{1\le m\le N}$ of functions $h^m:\Xd^m\to\R$ such that for all $m$ the function $h^m$ is symmetric in its $m$ entries (such as $\{F^{N,m}\}_{1\le m\le N}$).
\begin{enumerate}[---]
\item For $1\le m\le N$ and $P\subset\llbracket m\rrbracket$ with $P\ne\varnothing$, we define $h^P:\Xd^m\to\R$ as
$h^P(z_{\llbracket m\rrbracket}):=h^{\sharp P}(z_P)$, where~$\sharp P$ is the cardinality of $P$,
and for $P=\varnothing$ we set~$h^\varnothing:=0$ (as $h^m$ is not defined for $m<1$).
\smallskip\item For $P\subset\llbracket m\rrbracket$ with $P\ne\llbracket N\rrbracket$, we define $h^{P\cup\{*\}}:\Xd^{m+1}\to\R$ as
$h^{P\cup\{*\}}(z_{\llbracket m\rrbracket},z_*):=h^{\sharp P+1}(z_P,z_*)$,
and for $P=\llbracket N\rrbracket$ we set $h^{\llbracket N\rrbracket\cup\{*\}}=0$ (as $h^m$ is not defined for $m> N$).
\smallskip\item For $P\subset\llbracket m\rrbracket$ and $k,\ell \in P$, we define the operators
\begin{eqnarray*}
S_{k,\ell}h^P&:=&-K(x_k,x_\ell)\cdot\nabla_{v_k}h^P,\\
H_kh^{P\cup\{*\}}&:=&-\int_{\Xd}K(x_k,x_*)\cdot\nabla_{v_k}h^{P\cup\{*\}}\,\ddr z_*.
\end{eqnarray*}
\end{enumerate}
\end{defin}

With this notation at hand, taking partial integrals in the Liouville equation~\eqref{eq:Liouville_Langevin}, we are led to the following BBGKY hierarchy of coupled equations for marginals: for all $1 \le m \le N$, 
\begin{align}\label{eq:BBGKY_basic}
(\partial_t - L_0^{(m)}) F^{N,\llbracket m\rrbracket} = \kappa \frac{N-m}{N} \sum_{k=1}^m H_kF^{N,\llbracket m\rrbracket\cup\{*\}} + \frac{\kappa}{N} \sum_{k,\ell= 1}^m S_{k,\ell} F^{N,\llbracket m\rrbracket}, 
\end{align}
where the operator $L_0^{(m)}$ stands for the Kronecker sum of the Fokker-Planck operator $L_0$ on the $m$-particle space $L^2(\omega^{\otimes m})$,
\begin{equation}\label{eq:def-L0m}
L_0^{(m)} := \sum_{j=1}^m \Big( \Id^{\otimes m-j} \otimes L_0 \otimes \Id^{\otimes j - 1} \Big).
\end{equation}
In the long-time limit $t \uparrow \infty$, as the $N$-particle density $F^N$ converges to equilibrium $M^N$, we find that the equilibrium marginals~$\{M^{N,m}\}_{1\le m\le N}$ are a stationary solution of the same hierarchy~\eqref{eq:BBGKY_basic}.

\subsection{Gibbs relaxation}
\label{subsec:Gibbs_relax}

We show that the uniform-in-$N$ exponential relaxation to Gibbs equilibrium in~$L^2$ can be obtained as a direct estimate on the BBGKY hierarchy~\eqref{eq:BBGKY_basic}, using hypocoercivity in form of the ergodic estimates of Lemma~\ref{lem:ergodic-strong2}.

\begin{proof}[Proof of Proposition~\ref{prop:correlations_L2_Langevin}(ii)]
The proof is based on the analysis of the BBGKY hierarchy~\eqref{eq:BBGKY_basic}, restricted on the subspace where we can use hypocoercivity of the $m$-particle Fokker-Planck operator simultaneously in all $m$ directions.
For that purpose, we consider the following orthogonal projections on $L^2(\omega^{\otimes m})$, for $P\subset\llbracket m\rrbracket$,
\begin{equation}\label{eq:def-proj}
\pi_{P}^\bot:=\prod_{k\in P}\pi_k^\bot,\qquad\pi_k^\bot=\Id-\pi_k,\qquad\pi_kh^{\llbracket m\rrbracket}=M_0(z_k)\int_{\Xd} h^{\llbracket m\rrbracket} \, \ddr z_k.
\end{equation}
We split the proof into two steps.

\medskip
\step1 Tensorization of~\eqref{eq:estimL-dual}: for all $m\ge1$, $g_1,\ldots,g_m\in L^2_\loc(\R_+;L^2(\Xd)^d)$, and $t\ge0$,
\begin{equation}\label{eq:estim-divgsum}
\BigN\sum_{k=1}^m\int_0^te^{(t-s)L_0^{(m)}}\pi_{\llbracket m\rrbracket\setminus\{k\}}^\bot\Div_{v_k}(g_k(s))\,\ddr s\BigN_{L^2(\omega^{\otimes m})}^2\,\lesssim\,\sum_{k=1}^m\int_0^te^{-2mc_0(t-s)}\3g_k(s)\3_{L^2(\omega^{\otimes m})}^2\ddr s,
\end{equation}
where we emphasize that the multiplicative constant only depends on $d,\beta,K,A$ (not on $m$).
This is a direct extension of~\eqref{eq:estimL-dual}, where we note that the use of the equivalent norm defined in~\eqref{eq:equi-norm-m} allows to have a multiplicative factor independent of $m$.
We include a short proof for completeness. By definition~\eqref{eq:equi-norm-m} of the modified norms on the product, integrating by parts, and using that
\[e^{tL_0^{(m)}}=(e^{tL_0})^{\otimes m},\]
we find
\begin{multline*}
\BigN\sum_{k=1}^m\int_0^te^{(t-s)L_0^{(m)}}\pi_{\llbracket m\rrbracket\setminus\{k\}}^\bot\Div_{v_k}(g_k(s))\,\ddr s\BigN_{L^2(\omega^{\otimes m})}\\
=\sup_{\phi_1,\ldots,\phi_m} \sum_{k=1}^m\int_0^t\int_{\Xd^m} \omega^{\otimes m}g_k(s)\,\Big(\big((\nabla_{v_k}+\beta v_k)e^{(t-s)L_0^*}\phi_k\big)\otimes\bigotimes_{\ell:\ell\ne k}\big(e^{(t-s)L_0^*}\pi_\ell^\bot\phi_\ell\big)\Big)\,\ddr s,
\end{multline*}
where the supremum runs over $\phi_1,\ldots,\phi_m\in L^2(\omega)$ with $\3\phi_k\3_{L^2(\omega)}^*=1$ for all $1\le k\le m$. By the Cauchy-Schwarz inequality, this can be bounded as follows,
\begin{multline*}
\BigN\sum_{k=1}^m\int_0^te^{(t-s)L_0^{(m)}}\pi_{\llbracket m\rrbracket\setminus\{k\}}^\bot\Div_{v_k}(g_k(s))\,\ddr s\BigN_{L^2(\omega^{\otimes m})}^2
\le \bigg(\sum_{k=1}^m\int_0^te^{-2mc_0(t-s)}\3g_k(s)\3^2_{L^2(\omega^{\otimes m})}\bigg)\\
\times\sup_{\phi_1,\ldots,\phi_m}\bigg(\sum_{k=1}^m\int_0^te^{2mc_0(t-s)}\Big(\bigN(\nabla_{v_k}+\beta v_k)e^{(t-s)L_0^*}\phi_k\bigN_{L^2(\omega)}^{*}\prod_{\ell:\ell\ne k}\3e^{(t-s)L_0^*}\pi_\ell^\bot\phi_\ell\bigN_{L^2(\omega)}^{*}\Big)^2\,\ddr s\bigg).
\end{multline*}
Now appealing to~\eqref{eq:equiv-norm} (with constant $1$) and to~\eqref{eq:estimL-0} (with $L_0$ replaced by $L_0^*$ and using the equivalence of norms~\eqref{eq:equiv-norm}), the claim~\eqref{eq:estim-divgsum} follows.

\medskip
\step2 Conclusion.\\
From the BBGKY equations~\eqref{eq:BBGKY_basic} for marginals, recalling that marginals of the Gibbs equilibrium satisfy the stationary version of the same equations, we find for all $1\le m\le N$,
\begin{multline*}
(\partial_t-L_0^{(m)})(F^{N,m}-M^{N,m})
=\kappa\frac{N-m}N\sum_{k=1}^m H_k\big(F^{N,\llbracket m\rrbracket\cup\{*\}}-M^{N,\llbracket m\rrbracket\cup\{*\}}\big)\\[-4mm]
+\frac{\kappa}N\sum_{k,\ell=1}^mS_{k,\ell}\big(F^{N,\llbracket m\rrbracket}-M^{N,\llbracket m\rrbracket}\big).
\end{multline*}
Applying the projection $\pi_{\llbracket m\rrbracket}^\bot$ to both sides of this equation, noting that it commutes with $L_0^{(m)}$,
writing the Duhamel formula, and taking the norm, we get
\begin{multline*}
\bigN\pi_{\llbracket m\rrbracket}^\bot(F^{N,m}_t-M^{N,m})\bigN_{L^2(\omega^{\otimes m})}
\le
\bigN e^{tL_0^{(m)}}\pi_{\llbracket m\rrbracket}^\bot(F_\circ^{N,m}-M^{N,m})\bigN_{L^2(\omega^{\otimes m})}\\
+
\kappa \BigN \sum_{k=1}^m\int_0^te^{(t-s)L_0^{(m)}} \pi_{\llbracket m\rrbracket}^\bot H_k\big(F_s^{N,\llbracket m\rrbracket\cup\{*\}}-M^{N,\llbracket m\rrbracket\cup\{*\}}\big)\ddr s\BigN_{L^2(\omega^{\otimes m})}\\
+
\frac{\kappa}N\BigN\sum_{k,\ell=1}^m\int_0^t e^{(t-s)L_0^{(m)}}\pi_{\llbracket m\rrbracket}^\bot S_{k,\ell}\big(F_s^{N,\llbracket m\rrbracket}-M^{N,\llbracket m\rrbracket}\big)\ddr s\BigN_{L^2(\omega^{\otimes m})}.
\end{multline*}
Using~\eqref{eq:equiv-norm} to bound the first term, recalling the definition of $S_{k,\ell},H_k$, appealing to the result~\eqref{eq:estim-divgsum} of Step~1 to bound the last two terms, and using exchangeability, we are led to
\begin{multline}\label{eq:estim-prod}
\bigN\pi_{\llbracket m\rrbracket}^\bot(F_t^{N,m}-M^{N,m})\bigN_{L^2(\omega^{\otimes m})}^2
\,\le\,
Ce^{-2mc_0t}\bigN \pi_{\llbracket m\rrbracket}^\bot(F_\circ^{N,m}-M^{N,m})\bigN_{L^2(\omega^{\otimes m})}^2\\
+
C\kappa^2\int_0^tme^{-2mc_0(t-s)}\bigN\pi_{\llbracket m-1\rrbracket}^\bot(F_s^{N,m+1}-M^{N,m+1})\bigN_{L^2(\omega^{\otimes m+1})}^2\ddr s\\
+
\frac{C\kappa^2 m^2}{N^2} \int_0^tme^{-2mc_0(t-s)}\bigN\pi^\bot_{\llbracket m-2\rrbracket}(F_s^{N,m}-M^{N,m})\bigN_{L^2(\omega^{\otimes m})}^2\ddr s.
\end{multline}
Noting that
\begin{multline*}
\pi_{\llbracket m-2\rrbracket}^\bot F^{N,\llbracket m\rrbracket}=\pi_{\llbracket m\rrbracket}^\bot F^{N,\llbracket m\rrbracket}
+M_0(z_{m-1})\pi_{\llbracket m\rrbracket\setminus\{m-1\}}^\bot F^{N,\llbracket m\rrbracket\setminus\{m-1\}}\\
+M_0(z_{m})\pi_{\llbracket m-1\rrbracket}^\bot F^{N,\llbracket m-1\rrbracket}
-M_0(z_{m-1})M_0(z_m)\pi_{\llbracket m-2\rrbracket}^\bot F^{N,\llbracket m-2\rrbracket},
\end{multline*}
the above becomes
\begin{multline*}
\bigN\pi_{\llbracket m\rrbracket}^\bot(F_t^{N,m}-M^{N,m})\bigN_{L^2(\omega^{\otimes m})}^2
\,\le\,
Ce^{-2mc_0t}\bigN \pi_{\llbracket m\rrbracket}^\bot(F_\circ^{N,m}-M^{N,m})\bigN_{L^2(\omega^{\otimes m})}^2\\
+
C\kappa^2\int_0^tme^{-2mc_0(t-s)}\Big(\sum_{n=m-2}^{m+1}\bigN\pi_{\llbracket n\rrbracket}^\bot(F_s^{N,n}-M^{N,n})\bigN_{L^2(\omega^{\otimes n})}^2\Big)\,\ddr s.
\end{multline*}
From here, we conclude by a bootstrap argument. Given $C_0,T_0>0$, assume that we have for all $1\le m\le N$ and $t\in[0,T_0]$,
\begin{equation}\label{eq:bootstrap-ass}
\bigN\pi^\bot_{\llbracket m\rrbracket}(F^{N,\llbracket m\rrbracket}_t-M^{N,\llbracket m\rrbracket})\bigN_{L^2(\omega^{\otimes m})}^2 \le C_0^me^{-c_0t}.
\end{equation}
Inserting this into the above and using that $\int_0^tme^{-2mc_0(t-s)}e^{-c_0s}\ddr s\le c_0^{-1}e^{-c_0t}$, we would then deduce for all $1\le m\le N$ and $t\in[0,T_0]$,
\begin{multline*}
\bigN\pi_{\llbracket m\rrbracket}^\bot\big(F_t^{N,\llbracket m\rrbracket}-M^{N,\llbracket m\rrbracket}\big)\bigN_{L^2(\omega^{\otimes m})}^2
\,\le\,
Ce^{-2mc_0t}\bigN \pi_{\llbracket m\rrbracket}^\bot\big(F_\circ^{N,\llbracket m\rrbracket}-M^{N,\llbracket m\rrbracket}\big)\bigN_{L^2(\omega^{\otimes m})}^2\\
+
C\kappa^2(1+C_0+C_0^{-1}+C_0^{-2})C_0^me^{-c_0t}.
\end{multline*}
As we assume chaotic data~\eqref{eq:chaos-in}, $F_\circ^{N,m}=\mu_\circ^{\otimes m}$, the first right-hand side term is bounded by $CR^me^{-2mc_0t}$, for some constant $R$ depending on $d,\beta,K,A,\|\mu_\circ\|_{L^2(\omega)}$.
Choosing $C_0=4CR$, and $\kappa_0>0$ small enough so that $C\kappa_0^2(1+C_0+C_0^{-1}+C_0^{-2})\le\frac14$, we deduce that~\eqref{eq:bootstrap-ass} automatically holds with an additional prefactor $\frac12$ for all $1\le m\le N$, $t\in[0,T_0]$, and $\kappa\in[0,\kappa_0]$. We conclude that~\eqref{eq:bootstrap-ass} must hold with~$T_0=\infty$ for $\kappa\in[0,\kappa_0]$. Hence, by the equivalence of norms, for all $1\le m\le N$ and $t\ge0$,
\[\big\|\pi_{\llbracket m\rrbracket}^\bot(F_t^{N,\llbracket m\rrbracket}-M^{N,\llbracket m\rrbracket})\big\|_{L^2(\omega^{\otimes m})}\,\le\,C^me^{-c_0t/2}.\]
Finally, using an orthogonal expansion in form of
\[\|F_t^{N,m}-M^{N,m}\|_{L^2(\omega^{\otimes m})}^2=\sum_{P\subset\llbracket m\rrbracket}\|M_0\|_{L^2(\omega)}^{m-\sharp P}\big\|\pi_P^\bot(F_t^{N,P}-M^{N,P})\big\|_{L^2(\omega^{\otimes m})}^2,\]
the conclusion of Proposition~\ref{prop:correlations_L2_Langevin}(ii) follows (up to renaming $c_0$).
\end{proof}

\begin{rem}
The above proof holds unchanged if we only assume that initial data $F^{N,m}|_{t=0}=F^{N,m}_\circ$ satisfy $\|F^{N,m}_\circ\|_{L^2(\omega)}\le R^m$ for all $1\le m\le N$, for some $R>0$: in that case, the dependence on~$\|\mu_\circ\|_{L^2(\omega)}$ in the statement of Proposition~\ref{prop:correlations_L2_Langevin}(ii) would simply be replaced by dependence on this constant~$R$.
\end{rem}

\subsection{Correlation functions}
Recall that the two-particle correlation function $G^{N,2}$ is defined to capture the defect to propagation of chaos at the level of two-particle statistics, cf.~\eqref{eq:def-GN2}.
To describe finer corrections to mean field, we further consider higher-order correlation functions $\{G^{N,m}\}_{2\le m\le N}$, which are defined as suitable polynomial combinations of marginals of~$F^N$ in such a way that the full particle distribution~$F^N$ is recovered in form of a cluster expansion: using the notation of Definition~\ref{def:operator-L2-Langevin}, for all $1 \le m \le N$,
\begin{equation}\label{eq:cluster-exp0}
F^{N,m}\,=\,\sum_{\pi\vdash \llbracket m\rrbracket}\prod_{A\in\pi}G^{N,A},
\end{equation}
where $\pi$ runs through the list of all partitions of the index set $\llbracket m\rrbracket$, and where $A$ runs through the list of blocks of the partition $\pi$, $A\subset\llbracket m\rrbracket$.
As is easily checked, correlation functions are fully determined by prescribing~\eqref{eq:cluster-exp0} together with the `maximality' requirement
\begin{equation}\label{eq:integ-correl}
\int_{\Xd} G^{N,\llbracket m\rrbracket}\,\ddr z_\ell=0,\qquad\text{for all $1\le \ell \le m$ and $m\ge2$.}
\end{equation}
We find explicitly
\begin{eqnarray*}
G^{N,1}&=&F^{N,1},\\
G^{N,2}&=&F^{N,2}-(F^{N,1})^{\otimes2},\\
G^{N,3}&=&\Sym\big(F^{N,3}-3F^{N,2}\otimes F^{N,1}+2(F^{N,1})^{\otimes 3}\big),\\
G^{N,4}&=&\Sym\big(F^{N,4}-4F^{N,3}\otimes F^{N,1}-3F^{N,2}\otimes F^{N,2}+12F^{N,2}\otimes (F^{N,1})^{\otimes 2}-6(F^{N,1})^{\otimes 4} \big),
\end{eqnarray*}
where the symbol `$\Sym$' stands for the symmetrization of coordinates, and we can write more generally, for all $1\le m\le N$,
\begin{equation}\label{eq:def-cumGm}
G^{N,m}\,:=\,\sum_{\pi\vdash \llbracket m \rrbracket}(\sharp\pi-1)!(-1)^{\sharp\pi-1}\prod_{A\in\pi}F^{N,A},
\end{equation}
where we use a similar notation as in~\eqref{eq:cluster-exp0} and where $\sharp\pi$ stands for the number of blocks in a partition~$\pi$.
Similarly, we also define equilibrium correlation functions $\{C^{N,m}\}_{1\le m\le N}$ associated with~$M^N$ and its marginals.

\subsection{BBGKY hierarchy for correlations}
Starting from the BBGKY hierarchy~\eqref{eq:BBGKY_basic} for marginals of $F^N$, and recalling how marginals can be expanded in terms of correlations~\eqref{eq:cluster-exp0} and vice versa~\eqref{eq:def-cumGm}, we can derive a corresponding BBGKY hierarchy of equations for correlations.
We refer for instance to~\cite[Section~4]{Hess_Childs_2023} for the derivation: straightforward but lengthy algebraic manipulations are needed to collect the different factors in this hierarchy.

\begin{lem}[e.g.~\cite{Hess_Childs_2023}]\label{lem:hier-corr-Langevin}
Given $K\in L^\infty(\R^{2d})^d$, consider a weak solution $F^N\in C(\R_+;\Pc\cap\Ld^1(\Xd^N))$ of the Liouville equation~\eqref{eq:Liouville_Langevin}.
The associated correlation functions $\{G^{N,m}\}_{1 \le m\le N}$ satisfy the following hierarchy of equations: for all $1\le m\le N$,
\begin{multline*}
(\partial_t-L_0^{(m)})G^{N,m}
\,=\,\kappa\frac{N-m}{N} \sum_{k=1}^m H_k G^{N,\llbracket m \rrbracket \cup \{\ast\}}\\
- \kappa \sum_{k=1}^m \sum_{A \subset \llbracket m \rrbracket - \{k\}} \frac{m-1-\sharp A}{N} H_k\big(G^{N,A \cup \{k, \ast\}} G^{N,\llbracket m \rrbracket - \{k\} - A} \big)\\
+ \kappa \frac{N-m}{N} \sum_{k=1}^m \sum_{A \subset \llbracket m \rrbracket - \{k\}} H_k\big( G^{N,A \cup \{k\}} G^{N,\llbracket m \rrbracket \cup \{\ast\} - A - \{k\}} \big) \\
- \kappa \sum_{k=1}^m \sum_{A \subset \llbracket m\rrbracket - \{k\}} \sum_{B\subset \llbracket m\rrbracket - \{k\} - A} \frac{m-1-\sharp A-\sharp B}{N} H_k\big( G^{N,A \cup \{k\}} G^{N,B \cup \{*\}} G^{N,\llbracket m\rrbracket - A-B - \{k\}} \big) \\
+ \frac{\kappa}{N} \sum_{k\ne\ell}^m S_{k,\ell} G^{N,\llbracket m \rrbracket} + \frac{\kappa}{N} \sum_{k \ne \ell}^m \sum_{A \subset \llbracket m\rrbracket - \{k , \ell \}} S_{k, \ell}\big( G^{N,A \cup \{k\}} G^{N,\llbracket m\rrbracket - A - \{k\}}\big).
\end{multline*}
Letting $t \uparrow \infty$, equilibrium correlations $\{C^{N,m}\}_{1\le m\le N}$ are a stationary solution of this hierarchy.
\end{lem}

\subsection{Uniform-in-time propagation of chaos and correlation estimates}
\label{sec:corr_Langevin}
Optimal uniform-in-time estimates for the propagation of chaos were first obtained in~\cite{Lacker-LeFlem-23} for $K$ bounded. Here we give a proof of Proposition~\ref{prop:correlations_L2_Langevin}(i), which further establishes uniform-in-time estimates on correlations, and we recover uniform-in-time propagation of chaos as a direct consequence.
The proof is an adaptation of the recent work~\cite{Xie-24} for the overdamped dynamics, further using here hypocoercivity in the same spirit as in Section~\ref{subsec:Gibbs_relax} above.

\begin{proof}[Proof of Proposition~\ref{prop:correlations_L2_Langevin}(i)]
We split the proof into three steps.

\medskip
\step1 Proof of~\eqref{eq:correl_strong_Langevin}.\\
Starting point is the BBGKY hierarchy of equations for correlations in Lemma~\ref{lem:hier-corr-Langevin}.
Starting from the equation for $G^{N,m}$, applying the projection $\pi_{\llbracket m\rrbracket}^\bot$ defined in~\eqref{eq:def-proj}, writing the Duhamel formula, and taking (modified) norms, we get for all $1\le m\le N$,
\begin{multline*}
\bigN\pi_{\llbracket m\rrbracket}^\bot G_t^{N,m}-e^{tL_0^{(m)}}\pi_{\llbracket m\rrbracket}^\bot G^{N,m}_\circ\bigN_{L^2(\omega^{\otimes m})}
\,=\,\kappa\biggN \int_0^te^{(t-s)L_0^{(m)}}\bigg\{\frac{N-m}{N} \sum_{k=1}^m \pi_{\llbracket m\rrbracket}^\bot H_k G_s^{N,\llbracket m \rrbracket \cup \{\ast\}}\\
- \sum_{k=1}^m \sum_{A \subset \llbracket m \rrbracket - \{k\}} \frac{m-1-\sharp A}{N} \pi_{\llbracket m\rrbracket}^\bot H_k\big(G_s^{N,A \cup \{k, \ast\}} G_s^{N,\llbracket m \rrbracket - \{k\} - A} \big)\\
+ \frac{N-m}{N} \sum_{k=1}^m \sum_{A \subset \llbracket m \rrbracket - \{k\}} \pi_{\llbracket m\rrbracket}^\bot H_k\big( G_s^{N,A \cup \{k\}} G_s^{N,\llbracket m \rrbracket \cup \{\ast\} - A - \{k\}} \big) \\
- \sum_{k=1}^m \sum_{A \subset \llbracket m\rrbracket - \{k\}} \sum_{B\subset \llbracket m\rrbracket - \{k\} - A} \frac{m-1-\sharp A-\sharp B}{N} \pi_{\llbracket m\rrbracket}^\bot H_k\big( G_s^{N,A \cup \{k\}} G_s^{N,B \cup \{*\}} G_s^{N,\llbracket m\rrbracket - A-B - \{k\}} \big) \\
+ \frac{1}{N} \sum_{k\ne\ell}^m \pi_{\llbracket m\rrbracket}^\bot S_{k,\ell} G_s^{N,\llbracket m \rrbracket}
+ \frac{1}{N} \sum_{k \ne \ell}^m \sum_{A \subset \llbracket m\rrbracket - \{k , \ell \}} \pi_{\llbracket m\rrbracket}^\bot S_{k, \ell}\big( G_s^{N,A \cup \{k\}} G_s^{N,\llbracket m\rrbracket - A - \{k\}}\big)\bigg\}\ddr s\biggN_{L^2(\omega^{\otimes m})}.
\end{multline*}
Now appealing to~\eqref{eq:estim-divgsum} to estimate the different terms similarly as in~\eqref{eq:estim-prod}, we are led to
\begin{multline*}
\bigN\pi_{\llbracket m\rrbracket}^\bot G_t^{N,m}-e^{tL_0^{(m)}}\pi_{\llbracket m\rrbracket}^\bot G^{N,m}_\circ\bigN_{L^2(\omega^{\otimes m})}^2
\,\lesssim\,\kappa^2\int_0^tme^{-2mc_0(t-s)}\3G_s^{N,m+1}\3_{L^2(\omega^{\otimes m+1})}^2\,\ddr s\\
+\kappa^2\int_0^tme^{-2mc_0(t-s)}\bigg(\sum_{\ell=0}^{m-1}\binom{m-1}{\ell} \frac{m-\ell-1}{N}\3G_s^{N,\ell+2}\3_{L^2(\omega^{\otimes \ell+2})} \3G_s^{N, m-\ell-1}\3_{L^2(\omega^{\otimes m-\ell-1})}\bigg)^2\,\ddr s\\
+\kappa^2\int_0^tme^{-2mc_0(t-s)}\bigg(\sum_{\ell=0}^{m-1}\binom{m-1}\ell\3G_s^{N,\ell+1}\3_{L^2(\omega^{\otimes \ell+1})}\3G_s^{N,m-\ell}\3_{L^2(\omega^{\otimes m-\ell})} \bigg)^2\,\ddr s\\
+\kappa^2\int_0^tme^{-2mc_0(t-s)}\bigg(\sum_{\ell=0}^{m-1}\sum_{r=0}^{m-\ell-1}\binom{m-1}{\ell,r}\frac{m-\ell-r-1}{N} \3G_s^{N,\ell+1}\3_{L^2(\omega^{\otimes\ell+1})}\\
\hspace{6cm}\times\3G_s^{N,r+1}\3_{L^2(\omega^{\otimes r+1})}\3G_s^{N,m-\ell-r-1}\3_{L^2(\omega^{\otimes m-\ell-r-1})} \bigg)^2\,\ddr s \\
+\kappa^2\frac{m^2}{N^2}\int_0^tme^{-2mc_0(t-s)}\3G_s^{N,m}\3_{L^2(\omega^{\otimes m})}^2\,\ddr s\\
+\kappa^2\frac{m^2}{N^2}\int_0^tme^{-2mc_0(t-s)}\bigg(\sum_{\ell=0}^{m-2}\binom{m-2}\ell \3G_s^{N,\ell+1}\3_{L^2(\omega^{\otimes \ell+1})}\3 G_s^{N,m-\ell-1}\3_{L^2(\omega^{\otimes m-\ell-1})}\bigg)^2\,\ddr s.
\end{multline*}
Again, the use of modified norms allows to avoid $C^m$ prefactors. Note that the definition of correlation functions ensures $\pi_{\llbracket m\rrbracket}^\bot G^{N,m}=G^{N,m}$ for $2\le m\le N$, cf.~\eqref{eq:integ-correl}, and, of course, for $m=1$,
\[\pi_1^\bot G^{N,1}_t-e^{tL_0}\pi_1^\bot G_\circ^{N,1}=F^{N,1}_t-e^{tL_0}\mu_\circ=G^{N,1}_t-e^{tL_0}G_\circ^{N,1}.\]
Further using $\int_0^tme^{-2mc_0(t-s)}\,\ddr s\le c_0^{-1}$, and reorganizing the combinatorial factors, we obtain
\begin{multline}\label{eq:pre-est-GNm}
\bigN G_t^{N,m}-e^{tL_0^{(m)}} G^{N,m}_\circ\bigN_{L^2(\omega^{\otimes m})}
\,\lesssim\,\kappa\sup_{[0,t]}\3G^{N,m+1}\3_{L^2(\omega^{\otimes m+1})}\\
+\kappa\sum_{\ell=1}^{m}\binom{m-1}{\ell-1}\sup_{[0,t]}\3G^{N,\ell}\3_{L^2(\omega^{\otimes \ell})}\3G^{N,m+1-\ell}\3_{L^2(\omega^{\otimes m+1-\ell})}\\
+\frac{\kappa m}{N}\sum_{\ell=1}^{m-1}\sum_{r=1}^{m-\ell}\binom{m-2}{\ell-1,r-1} \sup_{[0,t]}\3G^{N,\ell}\3_{L^2(\omega^{\otimes\ell})}\3G^{N,r}\3_{L^2(\omega^{\otimes r})}\3G^{N,m+1-\ell-r}\3_{L^2(\omega^{\otimes m+1-\ell-r})}\\
+\frac{\kappa m}{N}\sup_{[0,t]}\3G^{N,m}\3_{L^2(\omega^{\otimes m})}
+\frac{\kappa m}{N}\sum_{\ell=1}^{m-1}\binom{m-2}{\ell-1}\sup_{[0,t]}\3G^{N,\ell}\3_{L^2(\omega^{\otimes \ell})}\3 G^{N,m-\ell}\3_{L^2(\omega^{\otimes m-\ell})}.
\end{multline}
From here, we conclude by a bootstrap argument as in~\cite{Xie-24}. Given $C_0,T_0>0$, assuming that we have for all $1\le m\le N$ and $t\in[0,T_0]$,
\begin{equation}\label{eq:bootstrap-correl}
\3G_t^{N,m}\3_{L^2(\omega^{\otimes m})}\,\le\,\frac{C_0(m-1)!}{m^2}N^{1-m},
\end{equation}
the above would then yield
\begin{multline*}
\bigN G^{N,m}_t-e^{tL_0^{(m)}}G^{N,m}_\circ\bigN_{L^2(\omega^{\otimes m})}
\,\lesssim\,\kappa\frac{C_0(m-1)!}{m^2}N^{1-m}\bigg(
\frac{m}{N}
+C_0\sum_{\ell=1}^{m}\frac{m^2}{\ell^2(m+1-\ell)^2}\\
+C_0\frac{m}{m-1}\sum_{\ell=1}^{m-1}\frac{m^2}{\ell^2(m-\ell)^2}
+C_0^2 \frac{m}{m-1} \sum_{\ell=1}^{m-1}\sum_{r=1}^{m-\ell}\frac{m^2}{\ell^2r^2(m+1-\ell-r)^2}\bigg),
\end{multline*}
and thus, after a direct estimation of the different sums,
\begin{equation*}
\bigN G^{N,m}_t-e^{tL_0^{(m)}}G^{N,m}_\circ\bigN_{L^2(\omega^{\otimes m})}
\,\lesssim\,\kappa(1+C_0+C_0^2)\frac{C_0(m-1)!}{m^2}N^{1-m}.
\end{equation*}
Starting from chaotic data~\eqref{eq:chaos-in}, this yields
\begin{equation*}
\bigN G^{N,m}_t\bigN_{L^2(\omega^{\otimes m})}
\,\le\,C\mathds1_{\{m=1\}}\|\mu_\circ\|_{L^2(\omega)}+C\kappa(1+C_0+C_0^2)\frac{C_0(m-1)!}{m^2}N^{1-m}.
\end{equation*}
Choosing $C_0:=4C\|\mu_\circ\|_{L^2(\omega)}$ and $\kappa_0$ small enough so that $C\kappa_0(1+C_0+C_0^2)<\frac14$, we deduce that~\eqref{eq:bootstrap-correl} automatically holds with an additional prefactor $\frac12$ for all $1\le m\le N$, $t\in[0,T_0]$, and $\kappa\in[0,\kappa_0]$.
We conclude that~\eqref{eq:bootstrap-correl} must hold with $T_0=\infty$ for $\kappa\in[0,\kappa_0]$. Hence, by the equivalence of the norms, the desired correlation estimates~\eqref{eq:correl_strong_Langevin} follow.

\medskip 
\step2 Proof of propagation of chaos~\eqref{eq:prop-est-Langevin} for $m=1$.\\
Comparing the BBGKY equation~\eqref{eq:BBGKY_basic} for $F^{N,1}$ and the mean-field equation~\eqref{eq:VFP} for $\mu$, we find
\[(\partial_t-L_0)(F^{N,1}-\mu)\,=\,\kappa H_1\Big(\frac{N-1}NF^{N,\{1\}}F^{N,\{*\}}-\mu(z_1)\mu(z_*)\Big)+\kappa\frac{N-1}NH_1G^{N,\{1\}\cup\{*\}}.\]
By the Duhamel formula with $F^{N,1}|_{t=0}=\mu_\circ$, we deduce
\begin{multline*}
\|F^{N,1}_t-\mu_t\|_{L^2(\omega)}\,\le\,\kappa \bigg\|\int_0^te^{(t-s)L_0}H_1\Big(\frac{N-1}NF_s^{N,\{1\}}F_s^{N,\{*\}}-\mu_s(z_1)\mu_s(z_*)\Big)\ddr s\bigg\|_{L^2(\omega)}\\
+\kappa\Big\|\int_0^te^{(t-s)L_0}H_1G^{N,\{1\}\cup\{*\}}_s\,\ddr s\Big\|_{L^2(\omega)},
\end{multline*}
and thus, recalling the definition of the operator $H_1$ and appealing to~\eqref{eq:estimL-dual},
\begin{multline*}
\|F^{N,1}_t-\mu_t\|_{L^2(\omega)}^2\,\lesssim\,\kappa^2\int_0^te^{-2c_0(t-s)} \Big\|\frac{N-1}N(F_s^{N,1})^{\otimes2}-\mu_s^{\otimes2}\Big\|^2_{L^2(\omega^{\otimes2})}\,\ddr s\\
+\kappa^2\int_0^te^{-2c_0(t-s)}\|G^{N,2}_s\|_{L^2(\omega^{\otimes2})}^2\,\ddr s.
\end{multline*}
Using the bound~\eqref{eq:correl_strong_Langevin} on correlations, as proven in Step~1, we deduce
\begin{equation*}
\|F^{N,1}_t-\mu\|_{L^2(\omega)}\,\lesssim\,N^{-1}+\kappa\sup_{[0,t]}\|F^{N,1}-\mu\|_{L^2(\omega)}.
\end{equation*}
For $\kappa$ small enough, the last term can be absorbed in the left-hand side and the conclusion~\eqref{eq:prop-est-Langevin} follows for $m=1$.

\medskip
\step3 Proof of propagation of chaos~\eqref{eq:prop-est-Langevin} for all $1 \le m \le N$.\\
Starting point is the cluster expansion~\eqref{eq:cluster-exp0}. Singling out the term corresponding to the partition $\pi=\{\{1\},\ldots,\{m\}\}$ in this expansion, we can write
\begin{equation}\label{eq:expand-FNm-mum}
F^{N,m}-\mu^{\otimes m}
\,=\,
(F^{N,1})^{\otimes m}-\mu^{\otimes m}
+\sum_{\pi\vdash\llbracket m\rrbracket\atop\sharp\pi<m}\prod_{A\in\pi}G^{N,A}.
\end{equation}
Rather than directly applying the correlation estimates~\eqref{eq:correl_strong_Langevin} proven in Step~1, some care is needed to ensure a merely exponential dependence on $m$ in~\eqref{eq:prop-est-Langevin}.
For that purpose, we first note that
\begin{equation}\label{eq:red-bnd}
m^mN^{1-m}\le 4e^{m}N^{-1},\qquad\text{for $2\le m\le N$,}
\end{equation}
so that the correlation estimates~\eqref{eq:correl_strong_Langevin} yield in particular
\[\|G^{N,m}_t\|_{L^2(\omega^{\otimes m})}\,\le\, C^mN^{-1},\qquad \text{for $2\le m\le N$}.\]
Now using this together with the result~\eqref{eq:prop-est-Langevin} for $m=1$ as proven in Step~2, the conclusion~\eqref{eq:prop-est-Langevin} follows from~\eqref{eq:expand-FNm-mum} for all $m$.
\end{proof}

Letting $t\uparrow\infty$, combining the above result with the convergence of $F^N$ (resp.~$\mu$) to equilibrium~$M^N$ (resp.~$M$) from Proposition~\ref{prop:correlations_L2_Langevin}~(ii) (resp. Lemma~\ref{lem:ergodic-strong2}), we obtain the following corresponding estimates on equilibrium correlations.

\begin{cor}\label{cor:correl-Langevin} 
Let $K \in L^\infty(\R^{2d})^d$ and let $A$ satisfy~\eqref{eq:confinement-A}. There exist $\kappa_0,C > 0$ (only depending on $d, \beta, K,A$) such that given $\kappa \in [0,\kappa_0]$ we have for all $1\le m\le N$,
\begin{eqnarray*}
\|C^{N,m}\|_{L^2(\omega^{\otimes m})}&\le&(Cm)^mN^{1-m},\\
\|M^{N,m}-M^{\otimes m}\|_{L^2(\omega^{\otimes m})}&\le&C^mN^{-1}.
\end{eqnarray*}
\end{cor}\smallskip

\section{Proof of main results}\label{sec:thm-Langevin}
This section is devoted to the proof of our main result, Theorem~\ref{thm:main-Langevin}. 
Henceforth, let $K,A$ be as in~\eqref{eq:KLinfty}--\eqref{eq:confinement-A}, let $\mu_\circ \in \Pc(\Xd)\cap L^2(\omega)$,
consider the unique global weak solution $F^N\in C(\R_+; \mathcal{P}(\Xd^N) \cap L^2(\omega^{\otimes N}))$ of the Liouville equation~\eqref{eq:Liouville_Langevin} with tensorized data~\eqref{eq:chaos-in}, and the unique weak solution $\mu \in C(\R_+; \mathcal{P}(\Xd) \cap L^2(\omega))$ of~\eqref{eq:VFP}.

\subsection{Relaxation of correlations}
We start with the following optimal estimates on the exponential relaxation to equilibrium for correlation functions, already proving the second part~\eqref{eq:cross-err-cor} of Theorem~\ref{thm:main-Langevin}.

\begin{prop}\label{prop:GGtildeCCtilde-Langevin}
There exist $c_0 > 0$ (only depending on $d, \beta, K,A$) and $\kappa_0>0$ (further depending on~$\|\mu_\circ\|_{L^2(\omega)}$) such that given $\kappa \in [0, \kappa_0]$ we have for all $1\le m \le N$ and $t \ge 0$,
\begin{equation*}
\| G^{N,m}_t -  C^{N,m}\|_{L^2(\omega^{\otimes m})} \,\le\,(Cm)^m N^{1-m} e^{-c_0t}.
\end{equation*}
\end{prop}

\begin{proof}
First note that the case $m =1$ follows from uniform-in-$N$ Gibbs relaxation in Proposition~\ref{prop:correlations_L2_Langevin}(ii).
We now consider $2 \le m \le N$.
Starting from the equation for $G^{N,m}$ in Lemma~\ref{lem:hier-corr-Langevin}, taking the difference with the corresponding stationary equation for $C^{N,m}$, and repeating the estimates leading to~\eqref{eq:pre-est-GNm}, now rather using that for $m\ge2$ we have
\[e^{2c_0t}\int_0^tme^{-2mc_0(t-s)}h(s)\ddr s\,\le\, c_0^{-1}\sup_{0\le s\le t}e^{2c_0s}h(s),\]
we obtain for all $2\le m\le N$ and $t\ge0$,
\begin{multline*}
e^{c_0t}\bigN G_t^{N,m}-C^{N,m}\bigN_{L^2(\omega^{\otimes m})}
\,\lesssim\,
e^{c_0t}\bigN e^{tL_0^{(m)}}(G^{N,m}_\circ-C^{N,m})\BigN_{L^2(\omega^{\otimes m})}\\
+\kappa\sup_{0\le s\le t}e^{c_0s}\bigN G^{N,m+1}_s-C^{N,m+1}\bigN_{L^2(\omega^{\otimes m+1})}\\
+\kappa\sum_{\ell=1}^{m}\binom{m-1}{\ell-1}\sup_{0\le s\le t}e^{c_0s}\bigN G_s^{N,\ell}\otimes G_s^{N,m+1-\ell}-C^{N,\ell}\otimes C^{N,m+1-\ell}\bigN_{L^2(\omega^{\otimes m+1})}\\
+\frac{\kappa m}{N}\sum_{\ell=1}^{m-1}\sum_{r=1}^{m-\ell}\binom{m-2}{\ell-1,r-1} \sup_{0\le s\le t}e^{c_0s}\bigN G_s^{N,\ell}\otimes G_s^{N,r}\otimes G_s^{N,m+1-\ell-r}-C^{N,\ell}\otimes C^{N,r}\otimes C^{N,m+1-\ell-r}\bigN_{L^2(\omega^{\otimes m+1})}\\
+\frac{\kappa m}{N}\sup_{0\le s\le t}e^{c_0s}\bigN G_s^{N,m}-C^{N,m}\bigN_{L^2(\omega^{\otimes m})}\\
+\frac{\kappa m}{N}\sum_{\ell=1}^{m-1}\binom{m-2}{\ell-1}\sup_{0\le s\le t}e^{c_0s}\bigN G_s^{N,\ell}\otimes G_s^{N,m-\ell}-C^{N,\ell}\otimes C^{N,m-\ell}\bigN_{L^2(\omega^{\otimes m})}.
\end{multline*}
Combining this with the correlation estimates of Proposition~\ref{prop:correlations_L2_Langevin}(i) and Corollary~\ref{cor:correl-Langevin}, we get
\begin{multline*}
e^{c_0t}\bigN G_t^{N,m}-C^{N,m}\bigN_{L^2(\omega^{\otimes m})}
\,\lesssim\,
e^{c_0t}\bigN e^{tL_0^{(m)}}(G^{N,m}_\circ-C^{N,m})\BigN_{L^2(\omega^{\otimes m})}\\
+\kappa\sup_{0\le s\le t}e^{c_0s}\bigN G_s^{N,m+1}-C^{N,m+1}\bigN_{L^2(\omega^{\otimes m+1})}\\
+\kappa\sum_{\ell=1}^{m}\binom{m-1}{\ell-1}(C\ell)^{\ell}N^{1-\ell}\sup_{0\le s\le t}e^{c_0s}\bigN G_s^{N,m+1-\ell}-C^{N,m+1-\ell}\bigN_{L^2(\omega^{\otimes m+1-\ell})}\\
+\frac{\kappa m}{N}\sum_{\ell=1}^{m-1}\sum_{r=1}^{m-\ell}\binom{m-2}{\ell-1,r-1}(C\ell)^\ell(Cr)^rN^{2-\ell-r} \sup_{0\le s\le t}e^{c_0s}\bigN G_s^{N,m+1-\ell-r}-C^{N,m+1-\ell-r}\bigN_{L^2(\omega^{\otimes m+1-\ell-r})}\\
+\frac{\kappa m}N\sup_{0\le s\le t}e^{c_0s}\bigN G_s^{N,m}-C^{N,m}\bigN_{L^2(\omega^{\otimes m})}\\
+\frac{\kappa m}{N}\sum_{\ell=1}^{m-1}\binom{m-2}{\ell-1}(C\ell)^\ell N^{1-\ell}\sup_{0\le s\le t}e^{c_0s}\bigN G_s^{N,m-\ell}-C^{N,m-\ell}\bigN_{L^2(\omega^{\otimes m-\ell})}.
\end{multline*}
Note that both in the third and fifth lines the norm of $G^{N,m}-C^{N,m}$ appears with a prefactor~$O(\kappa)$ and can thus be absorbed in the left-hand side for $\kappa$ small enough. Also recall that the result is already known to hold for $m=1$, which allows to get rid of some other terms. Further using Lemma~\ref{lem:ergodic-strong2}(ii) for the damping of initial data, recalling $G^{N,m}_\circ=0$, and using again Corollary~\ref{cor:correl-Langevin}, we are then led to
\begin{multline*}
\sup_{0\le s\le t}e^{c_0s}\bigN G_s^{N,m}-C^{N,m}\bigN_{L^2(\omega^{\otimes m})}
\,\lesssim\,
(Cm)^mN^{1-m}
+\kappa\sup_{0\le s\le t}e^{c_0s}\bigN G_s^{N,m+1}-C^{N,m+1}\bigN_{L^2(\omega^{\otimes m+1})}\\
+\kappa\sum_{\ell=2}^{m-1}\binom{m-1}{\ell-1}(C\ell)^{\ell}N^{1-\ell}\sup_{0\le s\le t}e^{c_0s}\bigN G_s^{N,m+1-\ell}-C^{N,m+1-\ell}\bigN_{L^2(\omega^{\otimes m+1-\ell})}\\
+\frac{\kappa m}{N}\sum_{\ell=1}^{m-1}\sum_{r=1}^{m-\ell-1}\binom{m-2}{\ell-1,r-1}(C\ell)^\ell(Cr)^rN^{2-\ell-r} \sup_{0\le s\le t}e^{c_0s}\bigN G_s^{N,m+1-\ell-r}-C^{N,m+1-\ell-r}\bigN_{L^2(\omega^{\otimes m+1-\ell-r})}.
\end{multline*}
From here, we conclude by a bootstrap argument.
Given $C_0,T_0>0$, assuming that we have for all $2\le m\le N$ and $t\in[0,T_0]$,
\begin{equation}\label{eq:relax-correl-boots}
\3G_t^{N,m}-C^{N,m}\3_{L^2(\omega^{\otimes m})}\,\le\,m!C_0^mN^{1-m}e^{-c_0t},
\end{equation}
the above would then yield, together with Stirling's formula, for all $2\le m\le N$ and $t\in[0,T_0]$,
\begin{equation*}
\bigN G_s^{N,m}-C^{N,m}\bigN_{L^2(\omega^{\otimes m})}
\,\le\,
m!C_0^mN^{1-m}e^{-c_0t}\Big((C/C_0)^m
+CC_0\kappa
+CC_0\kappa \sum_{\ell=1}^{m-1}\sum_{r=1}^{m-\ell-1}(C/C_0)^{\ell+r}\Big).
\end{equation*}
Choosing $C_0=2C$, noting that the double sum is then bounded by $\sum_{\ell,r\ge1}(C/C_0)^{\ell+r}=1$, and choosing~$\kappa_0$ small enough so that $2CC_0\kappa_0<\frac14$, we deduce that~\eqref{eq:relax-correl-boots} automatically holds with an additional prefactor $\frac12$ for all $2\le m\le N$, $t\in[0,T_0]$, and $\kappa\in[0,\kappa_0]$. We conclude that~\eqref{eq:relax-correl-boots} must hold with $T_0=\infty$ for $\kappa\in[0,\kappa_0]$. By the equivalence of the norms, the conclusion follows.
\end{proof}

\subsection{Cross error estimates for many-particle densities}
The following result shows that it suffices to prove the cross error estimate~\eqref{eq:cross-err} in Theorem~\ref{thm:main-Langevin} for $m=1$. This is obtained as a direct corollary of the above relaxation estimates for correlations, cf.~Proposition~\ref{prop:GGtildeCCtilde-Langevin}.

\begin{prop}\label{prop:higher_order_marginals_strong}
There exist $c_0 > 0$ (only depending on $d, \beta, K,A$) and $\kappa_0>0$ (further depending on~$\|\mu_\circ\|_{L^2(\omega)}$) such that given $\kappa \in [0, \kappa_0]$ we have for all $2\le m \le N$ and $t\ge0$,
\begin{equation*}
\|F^{N,m}_t - \mu_t^{\otimes m} - M^{N,m} + M^{\otimes m} \|_{L^2(\omega^{\otimes m})} \,\le\, C^mN^{-1}e^{-c_0t}+C^m\|F_t^{N,1}-\mu_t-M^{N,1}+M\|_{L^2(\omega)}.
\end{equation*}
\end{prop}

\begin{proof}
Starting point is the cluster expansion~\eqref{eq:cluster-exp0} in form of
\begin{equation*}
F^{N,m}_t-M^{N,m}
\,=\,
\sum_{\pi\vdash \llbracket m\rrbracket}\Big(\prod_{A\in\pi}G^{N,A}_t-\prod_{A\in\pi}C^{N,A}\Big).
\end{equation*}
Singling out the term corresponding to $\pi=\{\{1\},\ldots,\{m\}\}$ with $G^{N,1}=F^{N,1}$ and $C^{N,1}=M^{N,1}$, this allows to decompose
\begin{multline*}
F^{N,m}_t-\mu_t^{\otimes m}-M^{N,m}+M^{\otimes m}\\
\,=\,
\Big((F^{N,1}_t)^{\otimes m}-\mu_t^{\otimes m}-(M^{N,1})^{\otimes m}+M^{\otimes m}\Big)
+\sum_{\genfrac{}{}{0pt}{2}{\pi\vdash \llbracket m\rrbracket}{\sharp\pi<m}}\Big(\prod_{A\in\pi}G^{N,A}_t-\prod_{A\in\pi}C^{N,A}\Big).
\end{multline*}
We split the proof into two steps, examining the two right-hand side terms separately.

\medskip
\step1 Proof that for all $m\ge1$ and $t\ge0$,
\begin{multline}\label{eq:full_tensor_higher_order-Langevin}
\big\|(F^{N,1}_t)^{\otimes m}-\mu_t^{\otimes m}-(M^{N,1})^{\otimes m}+M^{\otimes m}\big\|_{L^2(\omega^{\otimes m})}\\
\,\le\,C^mN^{-1}e^{-c_0t}+C^m\|F_t^{N,1}-\mu_t-M^{N,1}+M\|_{L^2(\omega)}.
\end{multline}
The second-order difference of tensor products can be decomposed as
\begin{multline*}
(F_t^{N,1})^{\otimes m}-\mu_t^{\otimes m}-(M^{N,1})^{\otimes m}+M^{\otimes m}
\,=\,
\sum_{k=1}^m M^{\otimes k-1}\otimes\big(F_t^{N,1}-\mu_t-M^{N,1}+M\big)\otimes(M^{N,1})^{\otimes m-k}\\
+\sum_{1\le k<\ell\le m}M^{\otimes k-1}\otimes(F_t^{N,1}-\mu_t)\otimes(M^{N,1})^{\otimes \ell-k-1}\otimes(F_t^{N,1}-M^{N,1})\otimes(F_t^{N,1})^{\otimes m-\ell}\\
+\sum_{1\le k<\ell\le m} M^{\otimes k-1}\otimes(\mu_t-M)\otimes\mu_t^{\otimes \ell-k-1}\otimes(F_t^{N,1}-\mu_t)\otimes(F_t^{N,1})^{\otimes m-\ell}.
\end{multline*}
We deduce
\begin{multline*}
\big\|(F_t^{N,1})^{\otimes m}-\mu_t^{\otimes m}-(M^{N,1})^{\otimes m}+M^{\otimes m}\big\|_{L^2(\omega^{\otimes m})}\\
\,\le\,
C^m\|F_t^{N,1}-\mu_t-M^{N,1}+M\|_{L^2(\omega)}
+C^m\|F_t^{N,1}-\mu_t\|_{L^2(\omega)}\Big(\|F_t^{N,1}-M^{N,1}\|_{L^2(\omega)}+\|\mu_t-M\|_{L^2(\omega)}\Big)\\
\,\le\,
C^m\|F_t^{N,1}-\mu_t-M^{N,1}+M\|_{L^2(\omega)}
+C^m\|F_t^{N,1}-\mu_t\|_{L^2(\omega)}\|\mu_t-M\|_{L^2(\omega)},
\end{multline*}
and the claim~\eqref{eq:full_tensor_higher_order-Langevin} then follows from uniform-in-time propagation of chaos and mean-field relaxation, cf.~Proposition~\ref{prop:correlations_L2_Langevin} and Lemma~\ref{lem:ergodic-strong2}(i).

\medskip
\step2 Proof that for all $\pi\vdash \llbracket m\rrbracket$ with $\sharp\pi<m$ we have for all $t\ge0$,
\begin{equation}\label{eq:estim-many-Langevin}
\Big\|\prod_{A\in\pi}G^{N,A}_t-\prod_{A\in\pi}C^{N,A}\Big\|_{L^2(\omega^{\otimes m})}\,\le\,C^mN^{-1}e^{-c_0t}.
\end{equation}
Given $\pi\vdash \llbracket m\rrbracket$, we start from
\begin{multline*}
\Big\|\prod_{A\in\pi}G^{N,A}_t-\prod_{A\in\pi}C^{N,A}\Big\|_{L^2(\omega^{\otimes m})}
\,\le\,\sum_{A\in\pi}\|G^{N,\sharp A}_t-C^{N,\sharp A}\|_{L^2(\omega^{\otimes\sharp A})}\\[-3mm]
\times\prod_{B\in\pi\setminus\{A\}}\Big(\|G^{N,\sharp B}_t\|_{L^2(\omega^{\otimes\sharp B})}+\|C^{N,\sharp B}\|_{L^2(\omega^{\otimes\sharp B})}\Big).
\end{multline*}
Recalling~\eqref{eq:red-bnd}, we note that Proposition~\ref{prop:GGtildeCCtilde-Langevin} yields for all $1\le k\le N$,
\[\|G^{N,k}_t-C^{N,k}\|_{L^2(\omega^{\otimes k})}\,\le\,(Ck)^{k}N^{1-k}e^{-c_0t}\,\le\,C^{k}e^{-c_0t}\big(N^{-1}+\mathds1_{\{k=1\}}\big),\]
and similarly the correlation estimates of Proposition~\ref{prop:correlations_L2_Langevin}(i) and Corollary~\ref{cor:correl-Langevin} yield
\[\|G^{N,k}_t\|_{L^2(\omega^{\otimes k})}+\|C^{N,k}\|_{L^2(\omega^{\otimes k})}\,\le\,(Ck)^{k}N^{1-k}\,\le\,C^{k}\big(N^{-1}+\mathds1_{\{k=1\}}\big).\]
Inserting this into the above, we are led to
\begin{equation*}
\Big\|\prod_{A\in\pi}G^{N,A}_t-\prod_{A\in\pi}C^{N,A}\Big\|_{L^2(\omega^{\otimes m})}\\
\,\le\,C^me^{-c_0t}\prod_{A\in\pi}\big(N^{-1}+\mathds1_{\{\sharp A=1\}}\big).
\end{equation*}
As the condition $\sharp\pi<m$ ensures that there is $A\in\pi$ with $\sharp A\ge2$, the claim~\eqref{eq:estim-many-Langevin} follows.
\end{proof}

\subsection{Proof of Theorem~\ref{thm:main-Langevin}}
We are now in position to conclude the proof of Theorem~\ref{thm:main-Langevin}. Recall that the second part~\eqref{eq:cross-err-cor} already follows from Proposition~\ref{prop:GGtildeCCtilde-Langevin}, so it remains to prove~\eqref{eq:cross-err}.
Comparing the BBGKY equation~\eqref{eq:BBGKY_basic} for $F^{N,1}$, the mean-field equation~\eqref{eq:VFP} for $\mu$, and the corresponding stationary equations for steady states, we find
\begin{equation}\label{eq:cross-error}
(\partial_t - L_0) \big( F^{N,1} - \mu -  M^{N,1} + M \big) \,=\,
\kappa H_1\bigg(\frac{N-1}N\big(F^{N,\{1,*\}}-M^{N,\{1,*\}}\big)-\mu^{\{1\}}\mu^{\{*\}}+M^{\{1\}}M^{\{*\}}\bigg).
\end{equation}
By the Duhamel formula with $F^{N,1}|_{t=0}=\mu_\circ$, recalling the definition of the operator $H_1$, and appealing to Lemma~\ref{lem:ergodic-strong2}(ii) with $\pi^\bot(M^{N,1}-M)=M^{N,1}-M$, we deduce
\begin{multline*}
\| F^{N,1}_t - \mu_t -  M^{N,1} + M \|_{L^2(\omega)}^2 \,\lesssim\,
e^{-2c_0t}\|M^{N,1}-M\|_{L^2(\omega)}^2\\
+\kappa^2 \int_0^te^{-2c_0(t-s)} \Big\|\frac{N-1}N\big(F_s^{N,2}-M^{N,2}\big)-\mu_s^{\otimes2}+M^{\otimes2}\Big\|_{L^2(\omega^{\otimes2})}^2\ddr s,
\end{multline*}
and thus, reorganizing the last term,
\begin{multline*}
\| F^{N,1}_t - \mu_t -  M^{N,1} + M \|_{L^2(\omega)}^2 \,\lesssim\,
e^{-2c_0t}\|M^{N,1}-M\|_{L^2(\omega)}^2
+\frac{\kappa^2}{N^2} \int_0^te^{-2c_0(t-s)}\|\mu_s-M\|_{L^2(\omega)}^2\ddr s\\
+\kappa^2 \int_0^te^{-2c_0(t-s)}\|F_s^{N,2}-\mu_s^{\otimes2}-M^{N,2}+M^{\otimes2}\|_{L^2(\omega^{\otimes2})}^2\ddr s.
\end{multline*}
Appealing to Corollary~\ref{cor:correl-Langevin}, Lemma~\ref{lem:ergodic-strong2}(i), and Proposition~\ref{prop:higher_order_marginals_strong} with $m=2$, respectively, to estimate the three right-hand side terms, we are led to
\begin{equation*}
\| F^{N,1}_t - \mu_t -  M^{N,1} + M \|_{L^2(\omega)}
\,\lesssim\,
N^{-1}t^\frac12e^{-c_0t}
+\kappa \sup_{[0,t]}\|F^{N,1} - \mu -  M^{N,1} + M \|_{L^2(\omega)}.
\end{equation*}
For $\kappa$ small enough, the last term can be absorbed and the conclusion~\eqref{eq:cross-err} then follows for $m=1$ (up to redefining $c_0$). By Proposition~\ref{prop:higher_order_marginals_strong}, this is enough to conclude for all $1\le m\le N$.
\qed

\subsection{Proof of Corollary~\ref{cor:MFTorus-Langevin}}
We note that all our arguments can be repeated for the corresponding dynamics on the torus $\T^d$ with $A\equiv0$. Ergodic estimates for the corresponding Fokker-Planck operator~$L_0$ on the torus are given in~\cite{Bouin_2020} and the perturbative argument for exponential mean-field relaxation in Lemma~\ref{lem:ergodic-strong2}(i) is easily adapted. The conclusion of Corollary~\ref{cor:MFTorus-Langevin} then follows up by noticing that for interaction forces deriving from a potential the explicit expression~\eqref{eq:Gibbs_eq} of the Gibbs equilibrium~$M^N$ ensures $M^{N,1}=M$, so that the cross error reduces to $F^{N,1}-\mu-M^{N,1}+M=F^{N,1}-\mu$.
\qed

\section{Overdamped setting}\label{sec:overdamped}

In this section, we briefly discuss the adaptation of the results to the overdamped Langevin dynamics~\eqref{eq:overdamped}.
To streamline the exposition, we use the same notation as for the underdamped dynamics, simply adapting their meaning accordingly, cf.~\eqref{eq:notation-overdamped}.

\subsection{Mean-field relaxation}
We consider the following self-adjoint elliptic operator on the weighted space $L^2(\omega)$: for $h\in C^\infty_c(\R^d)$,
\begin{equation}\label{eq:lin-op-L0}
L_0h\,:=\,\triangle h+\Div(h\nabla A)=\Div((\nabla+\nabla A)h).
\end{equation}
The following ergodic estimates for the mean-field evolution~\eqref{eq:MKV} and for the linearized semigroup~$e^{tL_0}$ are immediate by standard parabolic theory. They constitute the counterpart of Lemma~\ref{lem:ergodic-strong2} in the overdamped setting.

\begin{lem}\label{lem:ergodic-strong}
Let $K\in L^\infty(\R^{2d})^d$ and let the confining potential $A$ satisfy~\eqref{eq:confinement-A}. Then there exist $\kappa_0,c_0>0$ (only depending on $d,K,A$) such that the following hold for all $\kappa \in [0,\kappa_0]$.
\begin{enumerate}[(i)]
\item Given $\mu_\circ \in \Pc(\R^d)\cap L^2(\omega)$, the solution $\mu \in C(\R_+; \Pc(\R^d)\cap L^2(\omega))$ of the mean-field equation~\eqref{eq:MKV} satisfies for all $t \ge 0$,
\begin{equation*}
\|\mu_t-M\|_{L^2(\omega)} \,\lesssim\, e^{-c_0 t} \|\mu_\circ-M\|_{L^2(\omega)}.
\end{equation*}
\item For all $h\in C^\infty_c(\R^d)$ with $\int_{\R^d}h=0$, we have for all $t \ge 0$,
\begin{equation*}
e^{2c_0t}\|e^{tL_0} h\|_{L^2(\omega)}^2+\int_0^t e^{2c_0s}\| (\nabla+\nabla A)e^{sL_0}h\|_{L^2(\omega)}^2 \, \ddr s
\,\lesssim\, \|h\|_{L^2(\omega)}^2.
\end{equation*}
In particular, by duality, for all $g\in L^2_\loc(\R_+;L^2(\R^d)^d)$,
\[\Big\|\int_0^te^{(t-s)L_0}\Div(g(s))\,\ddr s\Big\|^2_{L^2(\omega)}\,\lesssim\,\int_0^te^{-2c_0(t-s)}\|g(s)\|_{L^2(\omega)}^2\,\ddr s.\]
\end{enumerate}
\end{lem}

\begin{proof}
As $A$ is uniformly convex, cf.~\eqref{eq:confinement-A}, we recall that the probability measure $\propto e^{-A}$ satisfies a Poincar\'e inequality, see e.g.~\cite[Section A.19]{Villani_2009}:
in terms of the weight $\omega=e^{A}$, this reads as follows, for all $h\in C^\infty_c(\R^d)$,
\begin{equation}\label{eq:Gauss-Poinc}
\Big\|h-M_0\int_{\R^d}h\Big\|_{L^2(\omega)}\,\lesssim\,\|(\nabla+\nabla A) h\|_{L^2(\omega)},\qquad M_0\,:=\,\tfrac{\omega^{-1}}{\int_{\R^d}\omega^{-1}}.
\end{equation}
Using this, the conclusion easily follows similarly as in the proof of Lemma~\ref{lem:ergodic-strong2}. We skip the details for shortness.
\end{proof}

\subsection{BBGKY hierarchy}

In the present overdamped setting, the BBGKY hierarchy of equations satisfied by marginals and by correlation functions take the exact same form as in the underdamped setting, cf.~\eqref{eq:BBGKY_basic} and Lemma~\ref{lem:hier-corr-Langevin}, up to replacing the operators $H_k$, $S_{k,\ell}$ introduced in Definition~\ref{def:operator-L2-Langevin} by the following,
\begin{eqnarray}
S_{k,\ell} h^P &:=& - \Div_{x_k}(K(x_k, x_\ell) h^P), \nonumber\\
H_k h^{P \cup \{\ast\}}&:=&-\Div_{x_k} \Big( \int_{\R^d} K(x_k, x_*) h^{P \cup \{\ast\}}\,\ddr x_* \Big).\label{eq:redef-SH}
\end{eqnarray}
In this new setting, now using the ergodic estimates in Lemma~\ref{lem:ergodic-strong} instead of Lemma~\ref{lem:ergodic-strong2}, the proof of Proposition~\ref{prop:correlations_L2_Langevin} is immediately adapted. In this way, we slightly extend the result of~\cite{Xie-24} on uniform-in-time correlation estimates by treating the whole space with confinement, and we obtain a uniform-in-$N$ Gibbs relaxation result in $L^2$ that partially improves previous works on the topic.

\subsection{Proof of Theorem~\ref{thm:main-overdamped}}
Starting from uniform-in-time correlation estimates and from the BBGKY hierarchy of equations for correlations,
now appealing to the ergodic estimates of Lemma~\ref{lem:ergodic-strong} to perform the estimates,
the conclusion follows by repeating the exact same steps as in the proof of Theorem~\ref{thm:main-Langevin} in Section~\ref{sec:thm-Langevin}. We skip the
details for shortness.\qed

\subsection{Proof of Corollary~\ref{cor:MFTorus-Xie}}
As for Corollary~\ref{cor:MFTorus-Langevin}, all our arguments can be directly repeated in the corresponding setting on the torus $\T^d$ with $A\equiv0$. 
It only remains to check that for a translation-invariant interaction kernel $K(x,y)=K_0(x-y)$ we have $M^{N,1}=M$. On the one hand, for $\kappa$ small enough, the uniqueness of the mean-field equilibrium~$M$ ensures that it coincides with the Lebesgue measure on~$\T^d$. On the other hand, taking the first marginal in the steady-state Liouville equation for $M^N$, we find
\[\triangle M^{N,1}=\kappa\frac{N-1}N\Div\Big(\int_{\T^d} K_0(\cdot-y)M^{N,2}(\cdot,y)\,\ddr y\Big).\]
As by translation invariance of the system we have $M^{N,2}(x,y)= M^{N,2}(0,y-x)$, we find that
\[\int_{\T^d} K_0(\cdot-y) M^{N,2}(\cdot,y)\,\ddr y=\int_{\T^d} K_0(-y) M^{N,2}(0,y)\,\ddr y\]
is a constant. The equation for $M^{N,1}$ thus reduces to $\triangle  M^{N,1}=0$, and we can conclude that $M^{N,1}$ is indeed constant, hence $M^{N,1}=M$.\qed

\section{Beyond weak interactions}\label{sec:extension-app}
This section is devoted to the proof of Theorem~\ref{th:non-pert}. We consider Cases~1 and~2 described in Section~\ref{sec:extensions}, for the overdamped dynamics on $\T^d$ with $A\equiv0$.
As $\kappa$ is now allowed to be large and as arguments in weak norms further lead to losses of derivatives, we cannot appeal to direct bootstrap arguments as before. For this reason, the structure of the proof needs to be considerably reworked.

\subsection{Ergodic estimates for mean-field operators}\label{subsec:ergodic-est}
Instead of the elliptic operator~$L_0$ in~\eqref{eq:lin-op-L0}, we consider the full linearized McKean-Vlasov operator at mean-field equilibrium~$M$: for $h\in C^\infty(\T^d)$,
\begin{equation}\label{eq:def-Lop}
L_M h\,:=\,\triangle h-\kappa\Div(h( K\ast M))-\kappa\Div(M(K\ast h)).
\end{equation}
Using properties of $K$ and recalling~$M=1$, this actually means
\begin{equation*}
L_M h\,=\,\left\{\begin{array}{lll}
\triangle h&:&\text{in Case~1},\\
\triangle h+\kappa\triangle W\ast h&:&\text{in Case~2}.
\end{array}\right.
\end{equation*}
In both cases, the unique ergodicity of the McKean-Vlasov dynamics and the ergodic properties of the semigroup $\{e^{tL_M}\}_{t\ge0}$ are established in~\cite[Sections~3.6.2--3.6.3]{Delarue_Tse_21} (see in particular~\cite{Carrillo_2023} for Case~2). As in~\cite{BD_2024}, these can be extended to arbitrary negative Sobolev norms.

\begin{lem}[\cite{Carrillo_2023,Delarue_Tse_21,BD_2024}]\label{prop:estim-Vts-Wk1*}
In both Cases~1 and~2, there exists $c_0>0$ (only depending on $d,K$), such that the following hold for any $\kappa\ge0$.
\begin{enumerate}[(i)]
\item Given $\mu_\circ\in\Pc\cap C^\infty(\T^d)$, the solution $\mu\in C(\R_+;\Pc\cap C^\infty(\T^d))$ of the corresponding mean-field equation~\eqref{eq:MKV} on $\T^d$ satisfies for all $t\ge0$ and $\ell\ge1$,
\[\|\mu_t-M\|_{W^{-\ell,1}(\T^d)}\,\lesssim_\ell\,e^{-c_0t}\|\mu_\circ-M\|_{W^{-\ell,1}(\T^d)},\]
for some multiplicative factor only depending on $d,K,\ell,\kappa$.
\smallskip\item For all $h\in C^\infty(\T^d)$ with $\int_{\T^d}h=0$, we have for all $t \ge 0$ and $\ell\ge1$,
\begin{equation*}
\|e^{tL_M} h\|_{W^{-\ell,1}(\T^d)}\,\lesssim_\ell\, e^{-c_0t}\|h\|_{W^{-\ell,1}(\T^d)}.
\end{equation*}
\end{enumerate}
\end{lem}

\subsection{Uniform-in-time correlation estimates}
\label{subsec:control_correl}

Based on the ergodic properties in Lemma~\ref{prop:estim-Vts-Wk1*} (together with corresponding properties of the mean-field dynamics linearized at the mean-field solution itself), uniform-in-time propagation of chaos in weak norms was shown in~\cite{Delarue_Tse_21} using the master equation formalism, which we further pushed forward in~\cite{BD_2024} to establish correlation estimates.

\begin{prop}[\cite{Delarue_Tse_21,BD_2024}]\label{thm:correl} 
Given $\mu_\circ\in \Pc\cap C^\infty(\T^d)$,
consider the unique global solution $F^N\in C(\R_+;\Pc\cap C^\infty(\T^{dN}))$ of the corresponding Liouville equation~\eqref{eq:Liouville_overdamped} on $\T^{dN}$ with tensorized data~\eqref{eq:init_tensorization_overdamped}, and the unique solution $\mu\in C(\R_+;\Pc\cap C^\infty(\T^d))$ of the mean-field equation~\eqref{eq:MKV} on $\T^d$.
In both Cases~1 and~2, for any $\kappa\ge0$, we have for all $1\le m\le N$ and $t\ge0$,
\[\|F_t^{N,m}-\mu_t^{\otimes m}\|_{W^{-\ell_m,1}(\T^{dm})}\,\le\,C_mN^{-1},\]
and in addition,
\[\|G^{N,m}_t\|_{W^{-\ell_m,1}(\T^{dm})} \,\le\,  C_mN^{1-m},\]
for some $\ell_m$ only depending on
$m$, and some multiplicative factor $C_m$ further depending on $d,K,\mu_\circ,\kappa$.
\end{prop}

In the limit $t\uparrow\infty$, in view of the qualitative relaxation of $F^N$ to Gibbs equilibrium, together with mean-field relaxation, cf.\@ Lemma~\ref{prop:estim-Vts-Wk1*}(i), we deduce the following corresponding estimates for the correlation functions $\{C^{N,m}\}_{1\le m\le N}$ associated with Gibbs equilibrium $M^N$.

\begin{cor}\label{cor:correl} 
In both Cases~1 and~2, for any $\kappa\ge0$, we have for all $1\le m\le N$,
\begin{eqnarray*}
\|M^{N,m}-M^{\otimes m}\|_{W^{-\ell_m,1}(\T^{dm})}&\le&C_mN^{-1},\\
\|C^{N,m}\|_{W^{-\ell_m,1}(\T^{dm})}&\le&C_mN^{1-m}.
\end{eqnarray*}
\end{cor}

\subsection{Relaxation of correlations}
\label{subsec:Bogolyubov}

With the above tools at hand, we now turn to the proof of Theorem~\ref{th:non-pert}, and we start by adapting Proposition~\ref{prop:GGtildeCCtilde-Langevin}. Instead of a direct bootstrap argument, we now need to proceed by induction.
There is also another difference: uniform-in-$N$ Gibbs relaxation is no longer known to hold a priori (in contrast with Proposition~\ref{prop:correlations_L2_Langevin}(ii) in the small-$\kappa$ setting). We thus need to assume that some weak form of Gibbs relaxation holds a priori, which we formulate for convenience in terms of a cross error estimate, cf.~\eqref{eq:ass-cross1}.
Note that~\eqref{eq:ass-cross1} is already known to hold at least for $\theta_0=1$ as a consequence of Proposition~\ref{thm:correl} and Corollary~\ref{cor:correl}.

\begin{prop}
\label{prop:GGtildeCCtilde}
Assume that there exist $\theta_0,\ell_0\ge1$ such that for all $t\ge0$ and $N\ge1$,
\begin{equation}\label{eq:ass-cross1}
\|F^{N,1}_t-\mu_t-M^{N,1}+M\|_{W^{-\ell_0,1}(\T^d)}\,\lesssim\,N^{-1}e^{-c_0t}+N^{-\theta_0}.
\end{equation}
In both Cases~1 and~2, for any $\kappa\ge0$, we then have for all $2\le m\le N$ and $t\ge0$,
\begin{align}\label{eq:prop_G_m_full_diff}
\|G^{N,m}_t-C^{N,m}\|_{W^{-\ell_m,1}(\T^{dm})}\,\lesssim\,N^{1-m}e^{-c_0t}+N^{1-m-\theta_0},
\end{align}
for some $\ell_m$ only depending on $\ell_0,\theta_0,m$, and some multiplicative factor further depending on $d,K,\mu_\circ,\kappa$.
\end{prop}

\begin{proof}
Starting point is the BBGKY hierarchy of equations for correlations in Lemma~\ref{lem:hier-corr-Langevin}, which holds unchanged in the overdamped setting with the notation~\eqref{eq:redef-SH}. Note that it can be reorganized as follows, using the operator $L_M$ instead of $L_0$ in the left-hand side,
\begin{multline*}
(\partial_t-L_M^{(m)})G^{N,m}
\,=\,\kappa\frac{N-m}{N} \sum_{k=1}^m H_k G^{N,\llbracket m \rrbracket \cup \{\ast\}}\\
- \kappa \sum_{k=1}^m \sum_{A \subset \llbracket m \rrbracket - \{k\}} \frac{m-1-\sharp A}{N} H_k\big(G^{N,A \cup \{k, \ast\}} G^{N,\llbracket m \rrbracket - \{k\} - A} \big)\\
+ \kappa \frac{N-m}{N} \sum_{k=1}^m \sum_{A \subset \llbracket m \rrbracket - \{k\}\atop0<\sharp A<m-1} H_k\big( G^{N,A \cup \{k\}} G^{N,\llbracket m \rrbracket \cup \{\ast\} - A - \{k\}} \big) \\
- \kappa \sum_{k=1}^m \sum_{A \subset \llbracket m\rrbracket - \{k\}} \sum_{B\subset \llbracket m\rrbracket - \{k\} - A} \frac{m-1-\sharp A-\sharp B}{N} H_k\big( G^{N,A \cup \{k\}} G^{N,B \cup \{*\}} G^{N,\llbracket m\rrbracket - A-B - \{k\}} \big) \\
+ \frac{\kappa}{N} \sum_{k\ne j}^m S_{k,j} G^{N,\llbracket m \rrbracket} + \frac{\kappa}{N} \sum_{k \ne j}^m \sum_{A \subset \llbracket m\rrbracket - \{k , j \}} S_{k,j}\big( G^{N,A \cup \{k\}} G^{N,\llbracket m\rrbracket - A - \{k\}}\big)\\
+ \kappa  \sum_{k=1}^m H_k\bigg( \Big(\frac{N-m}{N} F^{N,\{k\}}-M^{\{k\}}\Big) G^{N,\llbracket m \rrbracket \cup \{\ast\} - \{k\}} + \Big(\frac{N-m}{N}F^{N,\{\ast\}} -M^{\{\ast\}}\Big)G^{N,\llbracket m\rrbracket}\bigg),
\end{multline*}
where as in~\eqref{eq:def-L0m} we use the notation $L_M^{(m)}$ for the Kronecker sum of operators
\begin{equation}\label{eq:def-LM(m)}
L_M^{(m)} \,:=\, \sum_{j=1}^m \Big( \Id^{\otimes m-j} \otimes L_M \otimes \Id^{\otimes j-1} \Big).
\end{equation}
Taking the difference with the corresponding stationary equation for $C^{N,m}$, writing the Duhamel formula, taking the norm in $W^{-\ell-1,1}$, and using the ergodic properties of $L_M$, cf.\@ Lemma~\ref{prop:estim-Vts-Wk1*}(ii), we get after straightforward simplifications, for all $1\le m\le N$ and $t,\ell\ge0$,
\begin{multline*}
\|G^{N,m}_t-C^{N,m}\|_{W^{-\ell-1,1}(\T^{dm})}
\,\lesssim_m\,
e^{-mc_0t}\|G^{N,m}_\circ-C^{N,m}\|_{W^{-\ell-1,1}(\T^{dm})}\\
+\int_0^te^{-mc_0(t-s)}\|G_s^{N,m+1}-C^{N,\llbracket m \rrbracket \cup \{\ast\}}\|_{W^{-\ell,1}(\T^{d(m+1)})}\ddr s\\
+\sum_{p=2}^{m-1}\binom{m-1}{p-1} \int_0^te^{-mc_0(t-s)}\big\|G_s^{N,p}\otimes G_s^{N,m+1-p}-C^{N,p}\otimes C^{N,m+1-p}\big\|_{W^{-\ell,1}(\T^{d(m+1)})}\ddr s\\
+\frac{1}{N}\sum_{p=1}^{m-1}\sum_{r=1}^{m-p}\binom{m-1}{p-1,r-1}\int_0^te^{-mc_0(t-s)}\big\| G_s^{N,p} \otimes G_s^{N,r}\otimes G_s^{N,m+1-p-r}\\[-2mm]
\hspace{9cm}-C^{N,p}\otimes C^{N,r}\otimes C^{N,m+1-p-r}\big\|_{W^{-\ell,1}(\T^{d(m+1)})}\ddr s\\
+\frac{1}{N} \sum_{p=1}^{m-1}\binom{m-2}{p-1} \int_0^te^{-mc_0(t-s)}\big\| G_s^{N,p}\otimes G_s^{N,m-p}-C^{N,p}\otimes C^{N,m-p}\big\|_{W^{-\ell,1}(\T^{dm})}\ddr s\\
+\frac{1}{N}\int_0^te^{-mc_0(t-s)}\|G_s^{N,m}-C^{N,m}\|_{W^{-\ell,1}(\T^{dm})}\ddr s\\
+\frac{1}{N} \int_0^te^{-mc_0(t-s)}\big\|G_s^{N,m}\otimes F_s^{N,1}-C^{N,m}\otimes M^{N,1}\big\|_{W^{-\ell,1}(\T^{d(m+1)})}\ddr s\\
+\int_0^te^{-mc_0(t-s)}\big\|(F_s^{N,1} -M)\otimes G_s^{N,m}-(M^{N,1} -M)\otimes C^{N,m}\big\|_{W^{-\ell,1}(\T^{d(m+1)})}\ddr s.
\end{multline*}
Note that there is indeed a loss of one derivative when estimating $H_k$ and $S_{k,j}$ in $W^{-\ell-1,1}$: in contrast with the $L^2$ estimates in previous sections, we can no longer exploit dissipation in $W^{-k,1}$.
From here, we argue by induction.
Assume that for some $\alpha\ge0$ we have for all $2\le m\le N$ and~$t\ge0$,
\begin{equation}\label{eq:ind-ass-GNmCNm}
\|G_t^{N,m}-C^{N,m}\|_{W^{-\ell_m,1}(\T^{dm})}\le C_m\Big( N^{1-m}e^{-c_0t}+N^{1-m-\theta_0}+N^{-\frac{m+\alpha}2}\Big),
\end{equation}
for some $\ell_m,C_m\ge0$. This estimate obviously holds with $\alpha=0$ as a consequence of correlation estimates in Proposition~\ref{thm:correl} and Corollary~\ref{cor:correl} with $\frac m2\le m-1$.
Inserting~\eqref{eq:ind-ass-GNmCNm} into the above, using the following decomposition for the last term,
\begin{multline*}
(F^{N,1} - M) \otimes G^{N,m} - (M^{N,1} - M) \otimes C^{N,m} \\
= (F^{N,1}-\mu) \otimes (G^{N,m}-C^{N,m}) + (\mu - M) \otimes (G^{N,m} - C^{N,m}) + (F^{N,1} - M^{N,1}) \otimes C^{N,m},
\end{multline*}
and appealing to Lemma~\ref{prop:estim-Vts-Wk1*}(i), Proposition~\ref{thm:correl}, and Corollary~\ref{cor:correl}, we deduce for all $2\le m\le N$ and~$t\ge0$,
\begin{multline*}
\|G^{N,m}_t-C^{N,m}\|_{W^{-\ell_m',1}(\T^{dm})}
\,\lesssim_m\,
N^{1-m}e^{-c_0t}
+N^{1-m-\theta_0}
+N^{-\frac{m+\alpha+1}2}\\
+N^{1-m}\int_0^te^{-mc_0(t-s)}\|F^{N,1} -M^{N,1}\|_{W^{1-\ell_m',1}(\T^d)},
\end{multline*}
for some $\ell_m'$ only depending on $\ell_m,m$, and some multiplicative factor also depending on $d,C_m,K,\mu_\circ,\kappa$. Now decomposing
\[F^{N,1} -M^{N,1}\,=\,(F^{N,1} -\mu-M^{N,1}+M)+(\mu-M),\]
and appealing to the induction assumption~\eqref{eq:ass-cross1} and to Lemma~\ref{prop:estim-Vts-Wk1*}(i), we get
\begin{equation*}
\|G^{N,m}_t-C^{N,m}\|_{W^{-\ell_m',1}(\T^{dm})}
\,\lesssim_m\,
N^{1-m}e^{-c_0t}
+N^{1-m-\theta_0}
+N^{-\frac{m+\alpha+1}2},
\end{equation*}
that is, \eqref{eq:ind-ass-GNmCNm} with $\alpha$ replaced by $\alpha+1$. The conclusion~\eqref{eq:prop_G_m_full_diff} follows by induction.
\end{proof}

\subsection{Cross error estimate for many-particle densities}\label{subsec:higher_order}

Based on relaxation estimates for correlations in Proposition~\ref{prop:GGtildeCCtilde},
the proof of Proposition~\ref{prop:higher_order_marginals_strong} can be immediately repeated in the present setting. We skip the details for shortness.

\begin{prop}\label{prop:higher_order_marginals_strong-re}
Assume that there exist $\theta_0,\ell_0\ge0$ such that for all $t\ge0$ and $N\ge1$,
\begin{equation}\label{eq:ass-cross1-re}
\|F^{N,1}_t-\mu_t-M^{N,1}+M\|_{W^{-\ell_0,1}(\T^d)}\,\lesssim\,N^{-1}e^{-c_0t}+N^{-\theta_0}.
\end{equation}
In both Cases~1 and~2, for any $\kappa\ge0$, we then have for all $2\le m \le N$ and $t\ge0$,
\begin{equation*}
\|F^{N,m}_t - \mu_t^{\otimes m} - M^{N,m} + M^{\otimes m}\|_{W^{-\ell_m,1}(\T^{dm})} \,\lesssim_m\, N^{-1}e^{-c_0t}+N^{-\theta_0},
\end{equation*}
for some $\ell_m > 0$ only depending on $\ell_0,\theta_0,m$,
and some multiplicative factor further depending on $d,K,\mu_\circ,\kappa$.
\end{prop}

\subsection{Proof of Theorem~\ref{th:non-pert}}
\label{subsec:conclusion}
Again, we need to avoid bootstrap arguments and we rather appeal to induction: assuming that there exist $\ell_0,\theta_0\ge0$ such that for all $t\ge0$ and $N\ge1$,
\begin{equation}\label{eq:thm-main-induc0}
\|F_t^{N,1}-\mu_t-M^{N,1}+M\|_{W^{-\ell_0,1}(\T^d)}\,\lesssim\,N^{-1}e^{-c_0t}+N^{-\theta_0},
\end{equation}
we shall show that the same automatically holds with $N^{-\theta_0}$ replaced by $N^{-\theta_0-1}$, up to increasing $\ell_0$; more precisely, for all $t \ge 0$ and $N\ge1$,
\begin{equation}\label{eq:thm-main-induc-pr0}
\|F_t^{N,1}-\mu_t-M^{N,1}+M\|_{W^{-\ell_0',1}(\T^d)}\,\lesssim\,N^{-1}e^{-c_0t}+N^{-\theta_0-1},
\end{equation}
for some $\ell_0'$ depending only on $\ell_0,\theta_0$.
Since~\eqref{eq:thm-main-induc0} with $\theta_0=1$ and some $\ell_0$ already follows from Proposition~\ref{thm:correl} and Corollary~\ref{cor:correl}, this would indeed yield the conclusion~\eqref{eq:cross-weakre} for $m=1$ by induction, and we recall that this is sufficient to conclude for all $2 \le m \le N$ by Proposition~\ref{prop:higher_order_marginals_strong-re}. Finally, the Gibbs relaxation result~\eqref{eq:relax-weak} follows from~\eqref{eq:cross-weakre} together with Lemma~\ref{prop:estim-Vts-Wk1*}.

We turn to the proof of~\eqref{eq:thm-main-induc-pr0} given~\eqref{eq:thm-main-induc0}.
Starting point is the BBGKY equation~\eqref{eq:cross-error} for the cross error, which can be reorganized as follows, after straightforward manipulations, using the operator $L_M$ instead of $L_0$ in the left-hand side,
\begin{multline*}
(\partial_t - L_M) \big( F^{N,1} - \mu -  M^{N,1} + M \big)\\
\,=\,
\kappa H_1\bigg(\frac{N-1}N\big(G^{N,\{1,*\}}-C^{N,\{1,*\}}\big)
-\frac{1}N\big(F^{N,\{1\}}F^{N,\{*\}}-M^{N,\{1\}}M^{N,\{*\}}\big)\\
+(F^{N,\{*\}}-\mu^{\{*\}}-M^{N,\{*\}}+M^{\{*\}})(F^{N,\{1\}}-\mu^{\{1\}})
+(\mu^{\{*\}}-M^{\{*\}})(F^{N,\{1\}}-\mu^{\{1\}})\\
+(M^{N,\{*\}}- M^{\{*\}})( F^{N,\{1\}} - \mu^{\{1\}} -  M^{N,\{1\}} + M^{\{1\}})
+(F^{N,\{*\}}-\mu^{\{*\}})(\mu^{\{1\}}-M^{\{1\}})
\bigg).
\end{multline*}
Writing the Duhamel formula, taking the norm in $W^{-\ell-1,1}$, and using the ergodic properties of $L_M$, cf.\@ Lemma~\ref{prop:estim-Vts-Wk1*}(ii), we get for all $t,\ell\ge0$,
\begin{multline*}
\big\| F^{N,1}_t - \mu_t -  M^{N,1} + M \big\|_{W^{-\ell-1,1}(\T^d)}
\,\lesssim\,
e^{-c_0t}\big\| F^{N,1}_\circ - \mu_\circ -  M^{N,1} + M \big\|_{W^{-\ell-1,1}(\T^d)}\\
+\int_0^te^{-c_0(t-s)}\bigg(\|G_s^{N,2}-C^{N,2}\|_{W^{-\ell,1}(\T^{2d})}
+\frac{1}N\|(F_s^{N,1})^{\otimes 2}-(M^{N,1})^{\otimes2}\|_{W^{-\ell,1}(\T^{2d})}\\
+\|\mu_s-M\|_{W^{-\ell,1}(\T^d)}\|F^{N,1}_s-\mu_s\|_{W^{-\ell,1}(\T^d)}\\
+\big\|F^{N,1}_s-\mu_s-M^{N,1}+M\big\|_{W^{-\ell,1}(\T^d)}\Big(\|F^{N,1}_s-\mu_s\|_{W^{-\ell,1}(\T^d)}+\|M^{N,1}- M\|_{W^{-\ell,1}(\T^d)}\Big)
\bigg)\ddr s.
\end{multline*}
Under the induction assumption~\eqref{eq:thm-main-induc0}, appealing to Lemma~\ref{prop:estim-Vts-Wk1*}(i), Proposition~\ref{thm:correl}, Corollary~\ref{cor:correl}, and Proposition~\ref{prop:GGtildeCCtilde}, this precisely leads us to~\eqref{eq:thm-main-induc-pr0}, and the conclusion follows.
\qed

\section*{Acknowledgements}
AB acknowledges financial support from the AAP ``Accueil'' from Universit\'e Lyon 1 Claude Bernard,
and MD from the European Union (ERC, PASTIS, Grant Agreement n$^\circ$101075879).\footnote{Views and opinions expressed are however those of the authors only and do not necessarily reflect those of the European Union or the European Research Council Executive Agency. Neither the European Union nor the granting authority can be held responsible for them.}

\bibliographystyle{plain}
\bibliography{biblio}

\end{document}